\documentclass{amsart}

\usepackage{mathrsfs}
\usepackage{mathtools}
\usepackage{verbatim}
\usepackage{tikz}
\usepackage{esint}

\usepackage[latin1]{inputenc}
\usepackage{amsfonts}
\usepackage{amsmath}
\usepackage{amsthm}
\usepackage{amssymb}
\usepackage{latexsym}
\usepackage{enumitem}

\newtheorem{theorem}{Theorem}
\newtheorem{lemma}[theorem]{Lemma}
\newtheorem{proposition}[theorem]{Proposition}
\newtheorem{corollary}[theorem]{Corollary}
\newtheorem{fact}[theorem]{Fact}

\newtheorem{question}[theorem]{Question}
\newtheorem{definition}[theorem]{Definition}

\theoremstyle{definition}

\numberwithin{equation}{section}

\begin{document}

\newcommand{\cc}{\mathfrak{c}}
\newcommand{\N}{\mathbb{N}}
\newcommand{\BB}{\mathbb{B}}
\newcommand{\C}{\mathbb{C}}
\newcommand{\Q}{\mathbb{Q}}
\newcommand{\R}{\mathbb{R}}
\newcommand{\Z}{\mathbb{Z}}
\newcommand{\T}{\mathbb{T}}
\newcommand{\st}{*}
\newcommand{\PP}{\mathbb{P}}
\newcommand{\rin}{\right\rangle}
\newcommand{\SSS}{\mathbb{S}}
\newcommand{\forces}{\Vdash}
\newcommand{\dom}{\text{dom}}
\newcommand{\osc}{\text{osc}}
\newcommand{\F}{\mathcal{F}}
\newcommand{\A}{\mathcal{A}}
\newcommand{\B}{\mathcal{B}}
\newcommand{\I}{\mathcal{I}}
\newcommand{\X}{\mathcal{X}}
\newcommand{\Y}{\mathcal{Y}}
\newcommand{\CC}{\mathcal{C}}

\thanks{The research of the first named author was supported by the NCN (National Science
Centre, Poland) research grant no.\ 2020/37/B/ST1/02613.}

\title[Countably tight dual ball with a nonseparable measure]
{Countably tight dual ball with a nonseparable measure}

\author{Piotr Koszmider}
\address{Institute of Mathematics  of the Polish Academy of Sciences,
ul. \'Sniadeckich 8,  00-656 Warszawa, Poland}
\email{\texttt{piotr.math@proton.me}}

\author{Zden\v ek  Silber}
\address{Institute of Mathematics of the Polish Academy of Sciences,
ul. \'Sniadeckich 8,  00-656 Warszawa, Poland}
\email{\texttt{zdesil@seznam.cz}}

\begin{abstract} 
We construct a compact Hausdorff space $K$ such that the space $P(K)$
of Radon probabiblity measures on $K$ considered with the weak$^*$ topology 
(induced from the space of continuous functions $C(K)$) is countably tight which
is a generalization of sequentiality (i.e., if a measure $\mu$ is in the closure
of a set $M$, there is a countable $M'\subseteq M$ such that
$\mu$ is in the closure of $M'$) but $K$ carries a Radon probability measure which has
uncountable Maharam type (i.e., $L_1(\mu)$ is nonseparable). 
The construction uses (necessarily) an additional set-theoretic assumption (the $\diamondsuit$ principle)
as it was already  known, by a result of Fremlin, that it is consistent that such spaces do not exist.

This should be compared with the result of Plebanek and Sobota who showed
that countable tightness of $P(K\times K)$ implies that all Radon measures on $K$
have countable type.  So, our example shows that
the tightness of $P(K\times K)$ and of $P(K)\times P(K)$ can be different
as well as $P(K)$ may have Corson property (C) while $P(K\times K)$ fails to have it answering a question of  Pol.
Our construction is also a  relevant example in the general context
of injective tensor products of Banach spaces complementing
recent results of Avil\'es,  Mart\'\i nez-Cervantes,
Rodr\'\i guez  and  Rueda Zoca.

\end{abstract}

\maketitle

\section{Introduction}
 
 Considerable attention was given in the literature to the question 
 about how big a Banach space $X$ (a compact space $K$) needs to be 
 to admit an isomorphic copy of $L_1(\mu)$ in $X^*$ for nonseparable $\mu$
 (to admit a nonseparable $\mu$) or how small must a nonseparable Banach space
  (a nonmetrizable compact space) be not to admit the existence of such measures.
 Certainly at least a partial motivation comes here from separable results:
 a compact $K$ continuously maps onto $[0,1]$ if and only if $K$ carries a nonatomic Radon measure,
 and a Banach space $X$ admits a subspace isomorphic to $\ell_1$ if and only
 if its dual $X^*$ admits a subspace isomorphic to $L_1([0,1])$ (\cite{pelczynski}).
 Also if $K$ is compact Hausdorff and metrizable, then  $L_1(\mu)$ is separable 
 for every Radon measure $\mu$ on $K$, moreover $L_1(\mu)$ does not isomorphically
  embed in $C(K)^*$ for $\mu$ nonseparable (this follows from 2.5 of \cite{haydon-l1}).
 
 It turned out that characterizatons of nonseparable analogues to these phenomena depend on
 additional set-theoretic axioms. For example, Pe\l czy\'nski's conjecture
 concerning Banach spaces (i.e., that
 $L_1(\{0,1\}^\kappa)$ embeds into $X^*$ if and only if $\ell_1(\kappa)$ embeds into $X$
 for a Banach space $X$ and each infinite cardinal $\kappa$) is independent from the usual axioms
 (see Section 4 of \cite{negrepontis})
 as well as the equivalence of a compact $K$ mapping onto $\{0,1\}^{\omega_1}$
 and carrying a nonseparable Radon measure.
 In fact, Martin's axiom {\sf MA} and the negation of the continuum hypothesis {\sf CH}
 imply such elegant characterizations (\cite{argyros-ma, fremlin-ma, plebanek-1st}) while {\sf CH} provides ``small'' nonmetrizable
 compact spaces (and corresponding Banach spaces $C(K)$) carrying
 nonseparable Radon measures. The classical three {\sf CH} small examples
 due to Haydon, Kunen and Talagrand (\cite{haydon-main, kunen-L, talagrand})
 were unified into one construction in Section 5 of \cite{negrepontis}.
 Many other striking consistent examples followed  (\cite{kunen-hs, plebanek-small, plebanek-app, plebanek-musing}).
 Papers \cite{kunen-ms, plebanek-large, negrepontis} offer some overviews of this research.
 In particular, whether first countable, Frechet-Urysohn or countably tight\footnote{The tightness of a 
 topological space $X$ is the minimal cardinal $\tau$ such that for every $A \subseteq X$ and $x \in \overline{A}$
 there is a subset $A_0$ of $A$ of cardinality at most $\tau$ such that $x \in \overline{A_0}$.} compact spaces
 may carry a nonseparable Radon measure is undecidable.
 
Another approach to understanding the property of embeddability
of nonseparable $L_1(\mu)$ into the dual could be to look at the topological properties (with respect to the weak$^*$ topology)
of the space of Radon measures $M(K)$ or probability measures $P(K)$ on $K$ instead
 of the topological properties of $K$ or isomorphic
properties of $C(K)^*$. The space $K$ homeomorphically embeds into compact $P(K)\subseteq M(K)$
and the weak$^*$ topology is smaller than the norm topology in the dual of a Banach space, so 
some natural requirements of being small concerning $P(K)$ are stronger than those on $K$ or $C(K)^*$ with the norm topology.
Here for example  the first countability
of $P(K)$ implies in {\sf ZFC}  that all measures on $K$ have countable type\footnote{By Theorem 3.5 of
\cite{dichotomy} and Proposition 2 of \cite{pol-sd}, if $P(K)$ is first countable, then each measure
 on $K$ is strongly countably determined.
On the other hand strongly countably determined measures have countable type (see \cite{damian}
page 167).}.
   
Plebanek and Sobota proved in \cite{damian} that if $P(K\times K)$ has countable tightness then $K$ does not carry
a Radon measure of uncountable type and asked if $K\times K$ can be replaced by $K$ in this result.
Our main result is the negative answer to this question:

\begin{theorem}\label{main} Assume $\diamondsuit$. 
There is a compact Hausdorff $K$ such that the space $P(K)$ of 
all probability Radon measures on $K$ has countable tightness in the weak$^*$ topology
but $K$ carries a Radon measure of uncountable type.
\end{theorem}
 
Our construction falls into  a general scheme of Haydon-Kunen-Talagrand constructions  (see Section 3)
taking into account a more elaborate type of the successor stage (see Section 5) motivated by 
constructions of Fedorchuk (e.g. \cite{fedorchuk}) and developed by the first-named author (e.g. \cite{few}, \cite{fewsur}).
As our space is another counterexample to Pe\l czy\'nski's conjecture and must produce $K$
that does not map continuously onto $\{0,1\}^{\omega_1}$, some type of additional set-theoretic hypothesis is necessary.
Here we use Jensen's $\diamondsuit$ principle which is a strengthening of {\sf CH} (see Section 4) particularly
handy at transfinite constructions where apparently too many requirements need to be  followed.
This principle, like here, has been used by Fajardo in
\cite{rogerio-quotients} and recently by G\l odkowski in \cite{dimension} in constructions of interesting Banach spaces of the form
$C(K)$ where one needs to take care of apparently too many sets of Radon measures.
We do not know if the hypothesis can be weakened to {\sf CH}, for example.

Based on the above mentioned result of Plebanek and Sobota our construction has also applications
to the question of preservation of properties by products:

\begin{corollary} Assume $\diamondsuit$. There is a compact Hausdorff space $K$ such that
$P(K)$ has countable tightness while $P(K\times K)$ has uncountable tightness.
\end{corollary}

Besides answering questions
from \cite{damian} this answers a question already asked by Pol in \cite{pol-sd} (Question B) long
before the result of Plebanek and Sobota. Here we should also recall a theorem of Malykhin (\cite{malykhin}) 
which says that countable tightness
is productive for compact spaces, so for example $P(K)\times P(K)$ for our $K$ has countable tightness.

Another context of our construction is  property (C) of Corson
which is a convex version of the Lindel\"of property:
 a Banach space $X$
has property (C)  if every collection of closed convex subsets of $X$ with empty intersection contains
a countable subcollection with empty intersection. Pol proved that it is equivalent to
a convex version of countable tightness of the dual ball of $X$. Namely, $X$ has property (C) if and only if 
whenever $x\in B_{X^*}$ is in the weak$^*$ closure of a set $A\subseteq B_{X^*}$, then
there is a  countable $A_0\subseteq A$ such that $x$ belongs to the weak$^*$ closed convex hull of $A_0$ (\cite{pol-c}).
Property (C) became a standard tool in nonseparable Banach space theory (\cite{fabian-etal, hajek-etal, kubis-book, zizler}),
for example it is a three space property (\cite{pol-c}) or it can be used
to characterize WLD Banach spaces as those spaces with property (C) which admits an $M$-basis (Theorem 3.3 of \cite{vwz}).
As noted in Theorem 5.6 of \cite{damian} the main result of \cite{damian} says that if $C(K\times K)$
has property (C), then $K$ does not carry a Radon measure of uncountable type. So Theorem \ref{main}
yields:

\begin{corollary}Assume $\diamondsuit$. There is a  compact Hausdorff $K$ such that
$C(K)$ has property (C) of Corson while $C(K\times K)$ does not have this property.
\end{corollary}
This answers questions asked by Pol in \cite{pol-sd} (Question B) and \cite{pol-c} (Section 5.3).
Both of the above corollaries require some additional assumption like $\diamondsuit$
as under {\sf MA} and the negation of {\sf CH} countable tighness of
$K$ (productive by Malykhin's theorem) already implies property (C) of $C(K)$ and countable tightness of $P(K)$
(this is a consequence of the above mentioned result of Fremlin (\cite{fremlin-ma} and Theorem 3.2 of \cite{fr-plebanek}).

The above results on $C(K\times K)$ can also be seen as particular cases of results on injective
tensor products of Banach spaces. Recently  Avil\'es, Mart\'\i nez-Cervantes, Rodr\'\i guez and Rueda Zoca
investigated property (C) from this point of view in Section 7 of \cite{tensor} (for example, generalizing
the above theorem of Plebanek and Sobota in Corollary 7.7). In fact,
by a result of Marciszewski and Plebanek (Proposition 2.4 of \cite{witek}) the existence of
a measure of uncountable type on $K$ is not only equivalent to the existence
of an isomorphic embedding of $L_1(\{0,1\}^{\omega_1})$
into $C(K)^*$ but also to the existence of a nonseparable measure on $B_{C(K)^*}$, so one can
look at the topics mentioned in this introduction in terms of measures on the compact $B_{C(K)^*}$
which is partially done in \cite{tensor}. In particular, it is noted in Section 2.7 of \cite{tensor} that a construction
attributed to Kalenda and presented as Corollary 4.4. in \cite{plebanek-convex} already provides under {\sf CH}
an example of a Banach space (not of the form $C(K)$) 
with property (C) whose injective tensor square does not have property (C).

Despite the above progress it is still open if countable tightness of $P(K)$ is equivalent in {\sf ZFC} to property (C) of
$C(K)$, a question already asked by Pol in \cite{pol-sd} (Question A) and in \cite{pol-c} (the end of Section 5.1.).
Recently, generalizing a result for $C(K)$ spaces from \cite{fr-plebanek}, Mart\'\i nez-Cervantes and Poveda
proved in \cite{pfa} that for every Banach space $X$ the Proper Forcing Axiom {\sf PFA} implies that countable tightness of
$B_{X^*}$ is equivalent to sequentiality of $B_{X^*}$ which is equivalent to property (C) of $X$.
We mention two more questions:
\vfill
\break
\begin{question}\label{questions}$ $
\begin{enumerate}
\item Is it consistent that there is a compact Hausdorff space $K$ 
such that $P(K)$ is hereditarily separable and $K$ carries a nonseparable Radon measure?
\item Is it consistent that there is a compact Hausdorff space $K$ 
such that $P(K)$ is sequential and $K$ carries a nonseparable Radon measure?
\end{enumerate}
\end{question}
Concerning Question \ref{questions} (1) note that  (a) the dual ball of any nonseparable Banach space is never
hereditarily Lindel\"of in the weak$^*$ topology (Proposition 6.4 of \cite{biorth}), (b) the positive answer
would strengthen Theorem \ref{main} as hereditarily separable spaces are countably tight, (c) in \cite{kunen-hs}
the principle $\diamondsuit$ was used to construct a hereditarily separable (and hereditarily Lindel\"of) compact $K$ which carries
a nonseparable measure.

 Concerning Question \ref{questions} (2), we remark
that it is a version of a similar question asked in \cite{damian} (Problem 5.10) where sequentiality is
replaced by the Frechet-Urysohn property. The positive answer to this question
would strengthen Theorem \ref{main} as sequential spaces are countably tight.

\section{Preliminaries}

\subsection{Notation and terminology}

Let us first fix some notation. We use the notation from the books \cite{fabian-etal} (Banach spaces), \cite{engelking} (topology) and \cite{semadeni} (measure theory, topology and Banach spaces). Let us point out the most used notation in detail.
The symbols $\N$ and $\R$ will denote the sets of positive integers a real numbers respectively. The symbol $\omega$ will denote the first infinite ordinal and will be identified with $\N \cup \{0\}$ for notational convenience, $\omega_1$ will denote the first uncountable cardinal. If $A$ and $B$ are sets, their symmetric difference is defined as $A \triangle B = (A \setminus B) \cup (B \setminus A)$.

All topological spaces in this paper are assumed to be Hausdorff and all Banach spaces are over the field $\R$. For a compact space $K$ we denote by $C(K)$ the Banach space of continuous functions on $K$ equipped with the supremum norm. Its dual space $C(K)^*$ is identified in the canonical way, by the Riesz representation theorem, with the Banach space $M(K)$ of signed Radon measures with the variation norm. 
We write $P(K)$ for the subset of $M(K)$ consisting of Radon probability measures on $K$. Unless otherwise stated, we always consider $M(K)$ and $P(K)$ equipped with the weak$^*$ topology (coming from the predual $C(K)$). The set $P(K)$ is then compact and contains a homeomorphic copy of $K$.

Let $A, B$ be sets and $f: A \rightarrow B$ a function. For a subset $A'$ of $A$ we define the restriction $f|A': A' \rightarrow B$ by setting $(f|A')(a) = f(a)$ for $a \in A'$. In particular, if $x \in \{0,1\}^{\omega_1}$ and $\alpha < \omega_1$, we define the restriction $x|\alpha \in \{0,1\}^{\alpha}$ by setting $(x|\alpha) (\beta) = x(\beta)$ for $\beta <\alpha$. 
If $\alpha < \omega_1$ and $x \in \{0,1\}^{\omega_1}$ we set $\pi_{\alpha,\omega_1}(x) = x|\alpha$.

For a measure $\mu$ we denote by $|\mu|$ its total variation. If $\mu$ and $\nu$ are two Radon measures on $K$, we say that they are singular with respect for each other, and write $\mu \perp \nu$, if there is a Borel subset $E$ of $K$ with $|\mu|(E) = |\nu|(K \setminus E) = 0$. Note that if $\mu \perp \nu$ and $\mu + \nu$ is positive, then both $\mu$ and $\nu$ are positive as well.

We say that a topological space $X$ has countable tightness if for every $A \subseteq X$ and $x \in \overline{A}$ there is a countable $B \subseteq A$ such that $x \in \overline{B}$. The following result is well known (see e.g. \cite{fr-plebanek}).

\begin{fact}
    Let $K$ be a compact space. Then the tightness of $B_{M(K)}$ is the same as the tightness of $P(K)$.
\end{fact}

Next we recall one of the equivalent definitions of the Maharam type of a measure. 
For more information, see \cite[Chapter 33]{fremlin-book} or \cite{dichotomy,damian,plebanek-large}.

\begin{definition}
    Let $K$ be a compact space. The Maharam type (or just type) of a 
    Radon probability measure $\mu \in P(K)$ is the density of the Banach space $L_1(\mu)$.

    Measures of countable Maharam type will be called separable and measures with uncountable Maharam type will be called nonseparable.
\end{definition}

\subsection{Baire sets}

As we will work with non-metrizable compact spaces, we will need to make the distinction
 between Borel sets and Baire sets. Let us recall the definition of Baire sets.

\begin{definition}
    A Baire subset of a compact set $L$ is an element of the $\sigma$-algebra generated by closed $G_\delta$-subsets of $L$. A Baire function $f: L\rightarrow \R$ is any function such that $f^{-1}[U]$ is a Baire set
    for any open $U\subseteq \R$.
\end{definition}

Clearly, if $L$ is metrizable (or more generally if every closed subset of $L$ is $G_\delta$, i.e. if 
$L$ is perfectly normal) then a set $A \subseteq L$ is Borel if and only if it is Baire. 
It is easy to see that in a totally disconnected compact space $L$ 
(which will be the case in our construction) Baire sets can be equivalently described as the elements of the 
$\sigma$-algebra generated by clopen subsets of $L$. 

The set of all bounded Baire functions on a compact space $L$ is the smallest set of functions on $L$
which contains continuous functions and is closed under pointwise limits of uniformly
bounded sequences (see \cite[Definition 3.3.6.]{dales} and the following discussion).
For information about Baire functions and Borel measurability in the abstract setting see \cite{spurny}.

Some known facts about Baire functions (which actually hold for real functions on any measurable space, 
see e.g. \cite{rudin}) we will use are the following: Every Baire function is a pointwise limit of a sequence
of simple Baire functions, moreover, a non-negative Baire function is a pointwise supremum of a 
sequence of simple Baire functions. A pointwise limit of a sequence of Baire functions is a Baire function.

It follows from the definition that for two positive Radon measures $\mu,\nu$ 
on some compact space $L$, if $\mu \perp \nu$, then there is a Borel subset $E$ 
of $L$ such that $\mu(E) = \nu(L\setminus E) = 0$. The next results will show us that we can choose $E$ to be not only Borel, but also Baire.

\begin{lemma} \label{lemma-Borel-Baire}
    Let $L$ be a compact space and let $\mu,\nu$ be positive Radon measures on $L$. 
    Then for every Borel subset $E$ of $L$ there is a Baire subset $E'$ of $L$ such that $\mu(E\triangle E') = \nu(E\triangle E')=0$.
\end{lemma}

\begin{proof}
    First suppose that $E$ is closed. Then by regularity there are open sets $E_n \supseteq E$, 
    $n \in \N$, such that $\nu(E_n) \rightarrow \nu(E)$ and $\mu(E_n) \rightarrow \mu(E)$. 
    Now, as $L$ is normal and $E$ is closed, there are closed $G_\delta$ sets $E_n'$, 
    $n \in \N$, such that $E \subseteq E_n' \subseteq E_n$. Then $E' = \bigcap_{n \in \N} E_n'$ is a Baire set and it is clear that $E'$ satisfies the conclusion of our lemma. Further, we have $E' \supseteq E$.

    Now, let $E$ be any Borel set. Let, for $n \in \N$, $E_n \subseteq E$ be compact such that
    $\nu(E_n) \rightarrow \nu(E)$ and $\mu(E_n) \rightarrow \mu(E)$. 
    Now we use the first part of the lemma to find Baire sets $E_n'$, $n \in \N$, 
    such that $E_n' \supseteq E_n$ and $\nu(E_n') = \nu(E_n)$ and $\mu(E_n') = \mu(E_n)$. 
    Set $E' = \bigcup_{n \in \N} E_n'$. Then $E'$ is Baire and satisfies the conclusion of the lemma.
\end{proof}

Immediately we get the following corollary:

\begin{corollary} \label{corollary-singular-Baire}
    Let $L$ be a compact space and $\mu,\nu$ be two mutually singular positive Radon measures on $L$. Then there is a Baire subset $E$ of $L$ such that $\mu(E) = \nu(L \setminus E) = 0$.
\end{corollary}

The reason we want to work with Baire functions instead of Borel functions is that, if defined 
on a closed subset of $\{0,1\}^{\omega_1}$, Baire functions depend only on countably many coordinates. 
In the following paragraphs we make this observation more precise.

\begin{definition}\label{def-depends}
    Suppose $X\subseteq L\subseteq\{0,1\}^{\omega_1}$ and 
    $f: L\rightarrow \R$ and $\alpha<\omega_1$. We say that the set
    $X$ depends on $\alpha$ (in $L$) if 
    $$X=\pi_{\alpha, \omega_1}^{-1}[\pi_{\alpha, \omega_1}[X]]\cap L.$$
    We say that the function $f$ depends on $\alpha$ (in $L$) if $f(x)=f(y)$ for any $x, y\in L$ such that $x|\alpha=y|\alpha$.
\end{definition}

Note that if a set $X$ or a function $f$ depends on $\alpha < \omega_1$ in $L$,
 it also depends on all $\alpha < \beta < \omega_1$ in $L$. The sets that depend 
 on $\alpha$ in $L$ are precisely the preimages under $\pi_{\alpha,\omega_1}$ intersected 
 with $L$. In other words, $X$ depends on $\alpha$ in $L$ if for every $x \in L$ we
  have that $x \in X$ if and only if $x|\alpha \in \pi_{\alpha,\omega_1}[X]$.  
  The fact that a function $f$ depends on $\alpha$ means that its value $f(x)$ 
  depends only on $x|\alpha$ for each $x \in L$. It is readily proved that a set 
  depends on $\alpha$ in $L$ if and only if its characteristic function depends on $\alpha$ in $L$.

In the next proposition we show that being a Baire function defined on a closed subset of
$\{0,1\}^{\omega_1}$ is the same as being Borel and depending on some $\alpha < \omega_1$.

\begin{proposition}\label{baire-depends} Suppose that $L\subseteq\{0,1\}^{\omega_1}$ is compact. A Borel function
$f: L\rightarrow \R$ is Baire if and only if  there is $\alpha<\omega_1$
such that $f$ depends on $\alpha$.
\end{proposition}

\begin{proof}
First prove the forward implication.
    Let $\mathcal{F}$ denote the set of functions $f$ on $L$ for which there is some $\alpha < \omega_1$ 
    such that $f$ depends on $\alpha$. Then clearly $\mathcal{F}$ is an algebra under pointwise operations. 
    Hence, we will be done if we show that
    \begin{enumerate}[label=(\roman*)]
        \item $\mathcal{F}$ is closed under taking pointwise limits of sequences;
        \item $\mathcal{F}$ contains characteristic functions of Baire sets;
    \end{enumerate}
    as every Baire function is a pointwise limit of simple Baire functions.

    To show (i) let $(f_n)_{n \in \N} \subseteq \mathcal{F}$ and $f$ be its pointwise limit. 
    Let, for $n \in \N$, $\alpha_n < \omega_1$ be such that $f_n$ depends on $\alpha_n$. 
    Then $f$ depends in $\alpha = \sup_n \alpha_n$: let $x,y \in L$ be such that 
    $x|\alpha = y|\alpha$. Then for $n \in \N$ also $x|\alpha_n = y|\alpha_n$, 
    and thus $f_n(x)=f_n(y)$. Hence, $f(x) = \lim f_n(x) = \lim f_n(y) = f(y)$.
    
    To show (ii) let $\mathcal{S}$ be the family of subsets $E$ of $L$ such that 
    $\chi_E \in \mathcal{F}$. Then $\mathcal{S}$ is clearly closed under taking complements. 
    It is also closed under countable intersections by (i) as 
    $\chi_{\bigcap_{n \in \N} E_n} = \lim_N \chi_{\bigcap_{n \leq N} E_n}$ 
    and $\chi_{\bigcap_{n \leq N} E_n} = \prod_{n \leq N} \chi_{E_n}$. 
    Hence, $\mathcal{S}$ is a $\sigma$-algebra. But $\mathcal{S}$ contains all 
    clopen sets as those are finite unions of finite intersections of clopen sets of 
    the form $\pi_{\alpha,\omega_1}^{-1}[E]$ for some $\alpha < \omega_1$ and clopen 
    $E \in \{0,1\}^\alpha$. It follows that $\mathcal{S}$ must contain all Baire sets.
    
    To prove the backward implication we observe that a Borel function $f$ on $L$
    which depends on $\alpha<\omega_1$ is equal to $f'\circ \pi_{\alpha, \omega_1}$
    for some Borel $f': \pi_{\alpha,\omega_1}[L]\rightarrow \R$. But
    $\pi_{\alpha,\omega_1}[L]$ is metrizable, and thus $f'$ is Baire. Hence, $f$ is a composition of a continuous map and a Baire map, and so is Baire as well. 
\end{proof}

The following corollary immediately follows from the previous proposition 
and the mentioned observation that a set depends on $\alpha$ if and only if its characteristic function does.

\begin{corollary} \label{Baire-set-depends}
    Suppose that $L \subseteq \{0,1\}^{\omega_1}$ is compact.  The set $E \subseteq L$ is a Baire set if and only if 
    there is $\alpha < \omega_1$ such that $E$ depends on $\alpha$.
\end{corollary}

In the following proposition we show that every integrable function 
(with respect to some Radon measure) is equal to a Baire function almost everywhere. 
This, in combination with Proposition \ref{baire-depends}, will later allow us to reduce 
the case from a general integrable function to a function that depends on a countable coordinate.

\begin{proposition}\label{baire-L1} Suppose that $L\subseteq\{0,1\}^{\omega_1}$ is compact,
$\mu$ is a Radon measure on $L$ and $f\in L_1(\mu)$. Then
there is a Baire function $f'$ such that $f' = f$ $\mu$-almost everywhere.
\end{proposition}

\begin{proof}
    First note that it follows from Lemma \ref{lemma-Borel-Baire} that for every simple Borel function
     $s$ there is a simple Baire function $s'$ such that $s'=s$ $\mu$-almost everywhere. 
     Now let $f \in L_1(\mu)$. We can assume that $f$ is non-negative, and so there
      is a sequence $(s_n)_{n \in \N}$ of simple Borel functions such that 
      $f(x) = \sup_n s_n(x)$ for all $x \in L$. Let $(s_n'')_{n \in \N}$ be simple 
      Baire functions such that, for all $n \in \N$, $s_n'' = s_n$ $\mu$-almost 
      everywhere. Set $M = \{x \in L: \; \sup_n s_n''(x) = \infty\}$. 
      Then $\mu(M) = 0$ (as $f$ is $\mu$-integrable, and thus finite 
      $\mu$-almost everywhere, and for $\mu$-almost every $x \in L$ 
      it holds that $x \in M$ implies $f(x) = \infty$). 
      Further, $M$ is a Baire set as $M = \bigcap_n \bigcup_k \{x \in L: s_k''(x) > n\}$ 
      and $(s_k'')_{k \in \N}$ are Baire functions. Hence, for $n \in \N$ the functions 
      $s_n' = s_n'' \cdot \chi_{L \setminus M}$ are Baire and $s_n' = s_n'' = s_n$ 
      $\mu$-almost everywhere. Let $f' = \sup_n s_n'$. 
      Then $f'$ is Baire and $f'(x)$ is finite for all $x \in L$ (for $x \in M$ 
      we have $f'(x) = 0$ and for $x \notin M$ we have $f'(x) = \sup_n s_n''(x) < \infty$ by
       the definition of $M$). Let, for $n \in \N$, $N_n$ be a $\mu$-null set such that $s_n' = s_n$ on
        $L \setminus N_n$, and let $N = \bigcup_n N_n$. Then $N$ is $\mu$-null and for $x \notin N$ 
        we have $f(x) = \sup_n s_n(x) = \sup_n s_n'(x) = f'(x)$. Hence, $f' = f$ $\mu$-almost everywhere and we are done.
\end{proof}

\subsection{Limits of compact spaces and measures}
 
 If $K=K_{\omega_1}\subseteq \{0,1\}^{\omega_1}$ is compact we will consider
 $K_\alpha=\pi_{\alpha, \omega_1}[K_{\omega_1}]$.  If $K$ is clear from the
 context in the rest of the paper, for $\beta < \alpha \leq \omega_1$, we consider $\pi_{\beta,\alpha}$ 
 to be defined only on $K_\alpha$ instead of the whole $\{0,1\}^\alpha$.
 
Given compact  $K=K_{\omega_1}\subseteq \{0,1\}^{\omega_1}$ we will identify $C(K_\beta)$ with a subspace of $C(K_\alpha)$ for $\beta < \alpha \leq \omega_1$
 by the canonical map induced by $\pi_{\beta,\alpha}$. More precisely, there is an injective map 
 $T: C(K_\beta) \rightarrow C(K_\alpha)$ that maps every $f \in C(K_\beta)$ to $T(f) = f \circ \pi_{\beta,\alpha} \in C(K_\alpha)$.
The dual operator $T^*: M(K_\alpha) \rightarrow M(K_\beta)$ to this identification is the ``restriction" 
of the corresponding functional on $C(K_\alpha)$ to $C(K_\beta)$, i.e., 
$T^* (\nu) (f) = \nu (f \circ \pi_{\beta,\alpha})$ for $\nu \in M(K_\alpha)$ and $f \in C(K_\alpha)$. 
We will further omit the operator $T$ in the notation and write $\nu|C(K_\beta)$ instead of 
$T^*(\nu)$.

 \begin{definition}\label{def-limits} Suppose that $\lambda$ is a limit ordinal and that for
 each $\beta<\lambda$ compact spaces  $K_\beta\subseteq\{0,1\}^\beta$ 
 and Radon measures $\mu_\beta$  on $K_\beta$ are such that 
 \begin{enumerate}
 \item  $\pi_{\beta,\alpha}[K_\alpha] = K_\beta$ for each $\beta < \alpha <\lambda$, 
 \item $\mu_\alpha(\pi_{\beta, \alpha}^{-1}[F])=\mu_\beta(F)$
 for each $\beta < \alpha<\lambda$ and $F \subseteq K_\beta$ Borel.
    \end{enumerate}
    Then the compact space $$K_\lambda=\{x\in \{0,1\}^\lambda: \forall \beta<\lambda \ x|\beta\in K_\beta\}$$
    is called the inverse limit of $(K_\beta)_{\beta < \lambda}$ and the unique Radon measure
    $\mu_\lambda$ on $K_\lambda$ satisfying
    $$\mu_\lambda(\pi_{\beta, \lambda}^{-1}[F])=\mu_\beta(F)$$
    for each $\beta < \lambda$ and each Borel $F \subseteq K_\beta$ is called the inverse limit of $(\mu_\beta)_{\beta < \lambda}$.
 \end{definition}

 In the next lemma we clarify the existence of the inverse limit $\mu_\lambda$.
 
 \begin{lemma}\label{limit} Suppose $\lambda$, $(K_\beta)_{\beta<\lambda}$ and $(\mu_\beta)_{\beta<\lambda}$
 are as in Definition \ref{def-limits}. Then the union $\bigcup_{\beta<\lambda}C(K_\beta)$ is norm dense in $C(K_\lambda)$ and there exists the unique 
 Radon measure $\mu_\lambda$ on the inverse limit $K_\lambda$ of $(K_\beta)_{\beta<\lambda}$ 
 which is the inverse limit of $(\mu_\beta)_{\beta<\lambda}$. Moreover 
 $\mu_\lambda|C(K_\beta)=\mu_\beta$ for every $\beta<\lambda$.
 \end{lemma}
 \begin{proof}
    We can consider $C(K_\beta)$, for 
    $\beta < \lambda$, as an increasing sequence of subspaces of $C(K_\lambda)$, whose union
     $\bigcup_{\beta<\lambda}C(K_\beta)$ is norm dense in $C(K_\lambda)$
      by the Stone-Weierstrass theorem. Indeed, it is a subalgebra of $C(K_\lambda)$
       (as every finite algebraic combination of $f_i \in C(K_{\alpha_i})$, $i=1\dots,n$, 
       is in $C(K_{\max\{\alpha_i; i \leq n\}})$), it separates points of $K_\lambda$ 
       (as for every $x \neq y \in K_\lambda$ there is $\beta < \lambda$ such 
       that $x(\beta) \neq y(\beta)$, and hence there is some $f \in C(K_\beta)$ 
       which separates $x|\beta$ and $y|\beta$, and thus $f \circ \pi_{\beta,\lambda}$ 
       separates $x$ and $y$), and it contains constant functions (as $1_{K_\lambda} = 1_{K_1} \circ \pi_{1,\lambda}$).
    The mapping $\mu_\lambda'=\bigcup_{\beta<\lambda}\mu_\beta$ 
    is then a norm one positive functional on $\bigcup_{\beta<\lambda}C(K_\beta)$, 
    and so by the Hahn-Banach theorem and density it uniquely extends to a norm one 
    positive functional $\mu_\lambda$ on $C(K_\lambda)$. 
    
    Further, we have $\mu_\lambda|C(K_\beta)=\mu_\beta$ for every $\beta<\lambda$.
    As two Radon measures on $K_\beta$ defining the same
    functional on $C(K_\beta)$  must be equal,  we conclude that  $\mu_\lambda(\pi_{\beta, \lambda}^{-1}[F])=\mu_\beta(F)$
 for each $\beta < \lambda$ and each Borel $F \subseteq K_\beta$.
    
\end{proof}

\section{General construction and its Radon measures}

In this Section we define the general steps of the our transfinite construction (Definition \ref{def-general})
of a compact $K$ and a measure $\mu$
up to some parameters $(F_\alpha, \phi_\alpha)_{\omega \leq \alpha < \omega_1}$. This follows the general
scheme used in the constructions of Haydon, Kunen or Talagrand mentioned in the Introduction.
We show that in this case the measure $\mu$ is always of uncountable Maharam type (Lemma \ref{lemma-unc-type}).
Moreover, if the construction satisfies Haydon's condition (Definition \ref{haydon-cond})
then as in \cite{haydon-main} all Radon measures on $K$ singular with respect to $\mu$ are determined
at some initial countable stage of the construction (Proposition \ref{singular-collapses}). This will be crucial for us 
and will allow us to employ
our successor step from Section 5 in the final construction in Section 6.

\begin{definition}\label{def-general} Let $\delta$ be an infinite ordinal.
We say that a transfinite sequence $(K_\alpha, \mu_\alpha)_{\alpha\leq \delta}$ is
determined by a parameters $(F_\alpha, \phi_\alpha)_{\omega \leq \alpha<\delta}$ if
the following conditions hold:

\begin{enumerate}[label=(C\arabic*)]
    \item For $\alpha < \omega$, $K_\alpha = \{0,1\}^\alpha$ and $\mu_\alpha$ is the product measure on $K_\alpha$.
   For limit $\lambda \leq\delta$ the space 
    $K_\lambda$ is the inverse limit of $(K_\beta)_{\beta<\lambda}$
   and  the measure 
    $\mu_\lambda$ is the inverse limit of  $(\mu_\beta)_{\beta<\lambda}$ as in Definition \ref{def-limits}. 
      \item For $\beta < \alpha \leq \delta$, $\pi_{\beta,\alpha}[K_\alpha] = K_\beta$.
    \item For $\beta < \alpha \leq \delta$ and $F \subseteq K_\beta$ Borel, 
    $\mu_\alpha(\pi_{\beta, \alpha}^{-1}[F])=\mu_\beta(F)$.
      \item For $\omega \leq \alpha < \delta$
      the set $F_\alpha$ is a closed nowhere dense subset of $K_\alpha$ with $\mu_\alpha(F_\alpha) 
    \geq \frac{3}{4}$. 
    \item For $\omega \leq \alpha < \delta$ the function $\phi_\alpha: K_\alpha \setminus F_\alpha \rightarrow \{0,1\}$ 
    is continuous.
    \item For $\omega \leq \alpha < \delta$ we have  $$K_{\alpha+1}=
    \{(x, \phi_\alpha(x)): x\in K_\alpha\setminus F_\alpha\}\cup (F_\alpha\times\{0,1\}).$$ 
    \item For $\omega \leq \alpha < \delta$ and every Borel $X\subseteq K_{\alpha+1}$ we have
        \begin{align*}
            \mu_{\alpha+1}(X)=&\mu_\alpha(\pi_{\alpha, \alpha+1}[X\setminus (F_\alpha\times\{0,1\})]) + \\
            &{1\over 2}\mu_\alpha(\pi_{\alpha, \alpha+1}[X\cap(F_\alpha\times\{0\})])+
            {1\over 2}\mu_\alpha(\pi_{\alpha, \alpha+1}[X\cap(F_\alpha\times\{1\})]).
        \end{align*}
\end{enumerate}
When the transfinite sequence is clear from the context we will write $K$ and $\mu$ for $K_\delta$ and $\mu_\delta$ respectively. 
\end{definition}

Note that conditions (C2) and (C3) from the above definition already follow from condition (C1). We decided to include them in the definition for later reference.

To see that the formula in condition (C7) is well defined we need to check that for every Borel $X \subseteq K_{\alpha+1}$ the sets $\pi_{\alpha, \alpha+1}[X\setminus (F_\alpha\times\{0,1\})]$, $\pi_{\alpha, \alpha+1}[X\cap(F_\alpha\times\{0\})]$ and $\pi_{\alpha, \alpha+1}[X\cap(F_\alpha\times\{1\})]$ are Borel in $K_\alpha$. But this is true as any such $X$ equals $(X_0 \times \{0\}) \cup (X_1 \times \{1\})$ for some Borel subsets $X_0$, $X_1$ of $K_\alpha$. Then our sets are, respectively, $(X_0 \cup X_1) \setminus F_\alpha$, $X_0 \cap F_\alpha$ and $X_1 \cap F_\alpha$, which are all Borel.

\begin{definition} \label{def-ext}
    The space $K_{\alpha+1}$ defined by (C6) and the measure $\mu_{\alpha+1}$ defined by (C7) will
     be called extensions of $K_\alpha$ and $\mu_\alpha$ by $\phi_\alpha$. 
\end{definition}

Note that conditions (C5) and (C6) ensure that the space $K_{\alpha+1}$ is compact.
The conditions (C6) and (C7) thus determine the construction of $K_{\alpha+1}$
 and $\mu_{\alpha+1}$ from $K_\alpha$ and $\mu_\alpha$ and imply that the 
 measure $\mu_{\alpha+1}$ is positive, provided that $\mu_\alpha$ is. 
 Our construction will be of length $\delta=\omega_1$ and the properties of the final compact space $K=K_{\omega_1}$ will 
 depend on the particular choices of $F_\alpha$ and $\phi_\alpha$, which will be specified later.

    For $\alpha < \omega_1$ the space $K_\alpha$ is metrizable, and hence we do not need to
     distinguish between Borel and Baire sets in $K_\alpha$. The distinction will become important
      in the non-metrizable space $K = K_{\omega_1}$.

\begin{lemma}\label{lemma-unc-type} 
    Suppose that a transfinite sequence $(K_\alpha, \mu_\alpha)_{\alpha\leq \omega_1}$ is determined by parameters $(F_\alpha, \phi_\alpha)_{\omega \leq \alpha<\omega_1}$. Then the measures $\mu_\alpha$ for $\omega\leq\alpha\leq\omega_1$ are atomless and  $\mu=\mu_{\omega_1}$ is of uncountable type.
\end{lemma}

\begin{proof}
  By the first parts of condition (C1)  the measure $\mu_\omega$ is atomless. If $\omega < \alpha \leq \omega_1$ 
    and $x \in K_\alpha$ we have  that
  $0 \leq \mu_\alpha(\{x\}) \leq \mu_\alpha (\pi_{\omega,\alpha}^{-1}[\{\pi_{\omega,\alpha}(x)\}])
      $ which by (C3) is equal to
      $\mu_\omega(\{\pi_{\omega,\alpha}(x)\}) = 0$ as required for the atomlessness of $\mu_{\alpha}$.

    To prove that $\mu=\mu_{\omega_1}$ is of uncountable type we will construct an 
    uncountable family $(A_\alpha)_{\alpha < \omega_1}$ of  closed Baire
    subsets of $K$, such that $\mu(A_\alpha) \geq \frac{3}{8}$ for $\alpha < \omega_1$ 
    and $\mu(A_\alpha \cap A_\beta) \leq \frac{1}{4}$ for $\alpha < \beta < \omega_1$. 
    Then
    \begin{align*}
        \|\chi_{A_\alpha} - \chi_{A_\beta}\|_{L_1(\mu)} = \mu(A_\alpha \triangle A_\beta) = 
        \mu(A_\alpha) + \mu (A_\beta) - 2 \mu (A_\alpha \cap A_\beta) \geq \frac{6}{8} - \frac{2}{4} = \frac{1}{4}
    \end{align*}
    for $\beta < \alpha < \omega_1$, and $L_1(\mu)$ is not separable which would
    prove that $\mu$ is  of uncountable type.

    Let $\alpha < \omega_1$ and set $A_\alpha = \pi_{\alpha+1,\omega_1}^{-1}[F_\alpha \times \{0\}]$
     for $\alpha < \omega_1$. We have
    \begin{align*}
        \mu(A_\alpha) = \mu(\pi_{\alpha+1,\omega_1}^{-1}[F_\alpha \times \{0\}]) 
        \overset{\text{(C3)}}{=} \mu_{\alpha+1} [F_\alpha \times \{0\}]\overset{\text{(C7)}}{=}
         \frac{1}{2} \mu_\alpha(F_\alpha) \overset{\text{(C4)}}{\geq} \frac{3}{8}. 
    \end{align*}
    Now, if $\beta < \alpha < \omega_1$,
    \begin{align*}
        \mu(A_\alpha \cap A_\beta) &= \mu \left(\pi_{\alpha+1,\omega_1}^{-1}[F_\alpha \times \{0\}] \cap 
        \pi_{\beta+1,\omega_1}^{-1}[F_\beta \times \{0\}] \right) \\
        &= \mu \left(\pi_{\alpha+1,\omega_1}^{-1} \left[(F_\alpha \times \{0\}) 
        \cap \pi_{\beta+1,\alpha+1}^{-1}[F_\beta \times \{0\}] \right]\right) \\
        &\overset{\text{(C3)}}{=} \mu_{\alpha+1}\left((F_\alpha \times \{0\}) 
        \cap \pi_{\beta+1,\alpha+1}^{-1}[F_\beta \times \{0\}]\right) \\
        &\leq \mu_{\alpha+1}\left((F_\alpha \times \{0\}) \cap (\pi_{\beta+1,\alpha}^{-1}[F_\beta \times \{0\}]
         \times \{0,1\})\right) \\
        &= \mu_{\alpha+1} \left((\pi_{\beta+1,\alpha}^{-1}[F_\beta \times \{0\}] \cap F_\alpha) \times \{0\}\right) \\
        &\overset{\text{(C7)}}{=} \frac{1}{2} \mu_{\alpha} \left(\pi_{\beta+1,\alpha}^{-1}[F_\beta \times \{0\}
        ] \cap F_\alpha\right) \\
        &\leq \frac{1}{2} \mu_{\alpha} \left(\pi_{\beta+1,\alpha}^{-1}[F_\beta \times \{0\}] \right) \\
        &\overset{\text{(C3)}}{=} \frac{1}{2} \mu_{\beta+1} (F_\beta \times \{0\}) \\
        &\overset{\text{(C7)}}{\leq} \frac{1}{2} \left( \frac{1}{2} \mu_{\beta} (F_\beta) \right)
        \leq \frac{1}{4}.
    \end{align*}
\end{proof}

\begin{definition}\label{haydon-cond} Suppose that a transfinite sequence $(K_\alpha, \mu_\alpha)_{\alpha\leq \omega_1}$ is
determined by parameters $(F_\alpha, \phi_\alpha)_{\omega \leq \alpha<\omega_1}$ 
as in Definition \ref{def-general}.  We say that it satisfies Haydon's condition if the following holds:
For any $\beta <\omega_1$ and any Borel $E\subseteq K_\beta$ such that 
    $\mu_\beta(E)=0$ there is $\beta\leq\alpha<\omega_1$ such that for every $\alpha\leq\gamma<\omega_1$
    we have
    $$\pi_{\beta, \gamma}^{-1}[E]\cap F_\gamma=\emptyset.$$

\end{definition}

Haydon's condition  thus says that Borel sets of $\mu_\beta$-measure $0$ for $\beta<\omega_1$
will eventually stop being split.
The following lemmas culminating in Proposition \ref{singular-collapses}  explain the importance of Haydon's condition.

\begin{lemma} \label{lemma-injective}
Suppose that a transfinite sequence $(K_\alpha, \mu_\alpha)_{\alpha\leq \omega_1}$ is
determined by parameters $(F_\alpha, \phi_\alpha)_{\omega \leq \alpha<\omega_1}$. 
    Suppose that $\alpha < \omega_1$ and $E \subseteq K_\alpha$ is Borel and satisfies 
    $\pi_{\alpha, \gamma}^{-1}[E]\cap F_\gamma=\emptyset$ for each $\alpha \leq \gamma < \omega_1$ . 
    Then for every $\alpha \leq \gamma \leq \omega_1$ the mapping 
    $\pi_{\alpha,\gamma} | \pi_{\alpha, \gamma}^{-1}[E]$ is injective.
\end{lemma}

\begin{proof}
    This follows by induction over $\gamma$. It clearly holds for $\gamma = \alpha$. 
    Suppose now that $\gamma = \beta + 1$ and the claim holds for $\beta$. 
    By the induction assumption, $\pi_{\alpha,\beta} | \pi_{\alpha, \beta}^{-1}[E]$ is injective. 
    Further, by the assumptions of the lemma $\pi_{\alpha, \beta}^{-1}[E]\cap F_\beta=\emptyset$. 
    As by (C6) the mapping $\pi_{\beta,\gamma}$ is injective on $\{(x,\phi_\beta(x)):x \in K_\beta \setminus F_\beta\}$, 
    the composition $\pi_{\alpha,\gamma} | \pi_{\alpha, \gamma}^{-1}[E] = 
    (\pi_{\alpha,\beta} \circ \pi_{\beta,\gamma}) | \pi_{\alpha, \gamma}^{-1}[E]$ is injective as well. 
    The induction step for limit ordinals is clear.
\end{proof}

\begin{definition}\label{def-nu-beta}
Suppose that a transfinite sequence $(K_\alpha, \mu_\alpha)_{\alpha\leq \omega_1}$ is
determined by parameters $(F_\alpha, \phi_\alpha)_{\omega \leq \alpha<\omega_1}$. 
    Suppose that $\alpha \leq \beta \leq \omega_1$, $\nu\in B_{C(K_\alpha)^*}$ is positive such that $\nu\perp\mu_\alpha$ and suppose
    a Borel $E\subseteq K_\alpha$ satisfies
    \begin{enumerate}
        \item $\nu(E)=\|\nu\|_{C(K_\alpha)^*}$,
        \item $\pi_{\alpha, \beta}|(\pi^{-1}_{\alpha, \beta}[E])$ is injective.
    \end{enumerate}
    Then by $\nu^\beta$ we mean a real valued function defined on Borel subsets of $K_\beta$ defined by
    $\nu^\beta(X)=\nu(\pi_{\alpha, \beta}[X \cap \pi_{\alpha,\beta}^{-1}[E]])$
    for each Borel $X\subseteq K_\beta$.
\end{definition}

We should note that $\nu^\beta$ is well defined. Indeed, for any Borel subset $X$ of $K_\beta$ the set $\pi_{\alpha, \beta}[X \cap \pi_{\alpha,\beta}^{-1}[E]]$ is Borel in $K_\alpha$ by the Lusin-Souslin theorem (see e.g. Theorem 15.1. of \cite{kechris}) as $\pi_{\alpha,\beta}$ is continuous and injective on the set $\pi_{\alpha,\beta}^{-1}[E]$, hence $\nu^\beta(X)$ is well defined. Let us also mention that if $\beta = \alpha + 1$, then $\nu^\beta(X) = \nu(\pi_{\alpha,\beta}[X])$ for any Borel $X \subseteq K_\beta$. In this case, $\pi_{\alpha,\beta}[X]$ is automatically Borel (as $X = (X_0 \times \{0\}) \cup (X_1 \cup \{1\})$ for some Borel subsets $X_0$, $X_1$ of $K_\alpha$, thus $\pi_{\alpha,\beta}[X] = X_0 \cup X_1$ is Borel). Further, $\nu(\pi_{\alpha,\beta}[X]) = \nu(\pi_{\alpha, \beta}[X \cap \pi_{\alpha,\beta}^{-1}[E]])$ by condition \textit{(1)}.

\begin{lemma}\label{unique-extension}
    Suppose that $(K_\alpha, \mu_\alpha)_{\alpha\leq \omega_1}$,
    $(F_\alpha, \phi_\alpha)_{\omega \leq \alpha<\omega_1}$,  $\nu, E, \alpha$ and $\beta$ are as in 
    Definition \ref{def-nu-beta}. Then $\nu^\beta$ is the unique positive Radon measure on $K_\beta$
    such that $\nu^\beta|C(K_\alpha)=\nu$.
\end{lemma}

\begin{proof}
    Let $\rho \in M(K_\beta)$ be positive and satisfy $\rho | C(K_\alpha) = \nu$. Recall that $\rho |C(K_\alpha) = \nu$ if and only if for all 
    $f \in C(K_\alpha)$ it is true that $\rho(f \circ \pi_{\alpha,\beta}) = \nu(f)$. We will prove the following two points:
    \begin{enumerate}[label=(\roman*)]
        \item Every $X \subseteq \pi_{\alpha,\beta}^{-1}[E]$ compact satisfies $\rho(X) = \nu(\pi_{\alpha,\beta}[X])$.
        \item Every $X \subseteq \pi_{\alpha,\beta}^{-1}[K_\alpha \setminus E]$ compact satisfies $\rho(X) = 0$. 
    \end{enumerate}
    The fact that $\rho(X)=\nu(\pi_{\alpha, \beta}[X \cap \pi_{\alpha,\beta}^{-1}[E]])$ for any $X\subseteq K_\beta$ Borel will then follow by regularity.
    
    We begin with (i) -- let $X \subseteq \pi_{\alpha,\beta}^{-1}[E]$ be compact. By metrizability of 
    $K_\alpha$ there are for each $n \in \mathbb{N}$ continuous functions $h_n \in S_{C(K_\alpha)}$ 
    such that $h_n$ pointwise converge to the characteristic function of 
    $\pi_{\alpha,\beta}[X]$ (as $\pi_{\alpha,\beta}[X]$ is $G_\delta$ by metrizability of $K_\alpha$). 
    But then also $h_n \circ \pi_{\alpha,\beta}$ pointwise converge to the characteristic function of 
    $X$ (as $\pi_{\alpha,\beta}|X$ is injective). Hence, by dominated convergence theorem,
    \begin{align*}
        \rho(X) = \lim \rho(h_n \circ \pi_{\alpha,\beta}) = \lim \nu(h_n) = \nu(\pi_{\alpha,\beta}[X]).
    \end{align*}
    Now we prove (ii). Let $X \subseteq \pi_{\alpha,\beta}^{-1}[K_\alpha \setminus E]$
     be compact and take for $n \in \mathbb{N}$ nonnegative functions $h_n \in S_{C(K_\alpha)}$ 
     such that $h_n|\pi_{\alpha,\beta}[X]=1$ and $h_n$ pointwise converge to the characteristic
      function of $\pi_{\alpha,\beta}[X]$. Then by positivity and dominated convergence theorem we have for each $n \in \mathbb{N}$
    \begin{align*}
        0 \leq \rho(X) \leq 
        \rho(h_n \circ \pi_{\alpha,\beta}) = \nu(h_n) \rightarrow \nu(\pi_{\alpha,\beta}[X]) = 0
    \end{align*}
    as $\pi_{\alpha,\beta}[X] \cap E = \emptyset$.
\end{proof}

In the following proposition we prove an analogue of a property of the main example of \cite{haydon-main} described in Theorem 3.1  of \cite{haydon-main}. Namely, 
we show that every positive Radon measure $\nu$ on $K$ which is singular with respect to 
$\mu$ is determined by its restriction to some countable ordinal. More precisely:

\begin{proposition}\label{singular-collapses}
Suppose that a transfinite sequence $(K_\alpha, \mu_\alpha)_{\alpha\leq \omega_1}$ is
determined by parameters $(F_\alpha, \phi_\alpha)_{\omega \leq \alpha<\omega_1}$ and satisfies
Haydon's condition.
    Let
    $\nu$ be a positive Radon measure on $K$ which is singular with respect to $\mu$.
    Then there is $\beta<\omega_1$ such that $\nu|C(K_{\beta'})=(\nu|C(K_\beta))^{\beta'}$
    for every $\beta\leq\beta'\leq\omega_1$.
\end{proposition}

\begin{proof}
    Without loss of generality we assume that $\nu(K) = 1$.
    By Corollary \ref{corollary-singular-Baire} 
    there is a Baire set $E'\subseteq K$ such that
    $\mu(E')=0$ and $\nu(E')=1$. Since $E'$ is Baire, by Corollary \ref{Baire-set-depends} there is  $\beta<\omega_1$
    such that $E'=\pi_{\beta, \omega_1}^{-1}[E]$ for some $E\subseteq K_\beta$.
     By (C3) we have that $\mu_\beta(E)=\mu(E')=0$
     and there is
     $\beta < \omega_1$ such that $\pi_{\beta,\omega_1}|E'$ is injective.
    Hence, $\pi_{\beta,\beta'}|(\pi_{\beta',\omega_1}[E'])$ is injective for every
     $\beta\leq\beta'\leq\omega_1$ as well and $(\nu|C(K_\beta))^{\beta'}$ is well defined.
    Fix $\beta\leq\beta'\leq\omega_1$. By Lemma \ref{unique-extension}
     it is enough to show that $(\nu|C(K_{\beta'}))|C(K_{\beta}) = \nu|C(K_{\beta})$. But by definition, for $f \in C(K_{\beta})$,
    \begin{align*}
        \left( (\nu|C(K_{\beta'}))|C(K_{\beta}) \right) (f) &= (\nu|C(K_{\beta'})) (f \circ \pi_{\beta,\beta'}) 
        = \nu (f \circ \pi_{\beta,\beta'} \circ \pi_{\beta',\omega_1}) \\
        &= \nu(f \circ \pi_{\beta,\omega_1}) = (\nu|C(K_\beta))(f).
    \end{align*}
\end{proof}

The following lemma will not be used in the following parts of the paper but it sheds light on the
weak$^*$ topology of $P(K)$.

\begin{lemma} \label{lemma-countable-character}
    Suppose that a transfinite sequence $(K_\alpha, \mu_\alpha)_{\alpha\leq \omega_1}$ is
    determined by parameters $(F_\alpha, \phi_\alpha)_{\omega \leq \alpha<\omega_1}$ and satisfies
    Haydon's condition.
    Suppose that $\nu\in P(K)$ and $\nu\perp\mu$.  Then the character of $\nu$ in $P(K)$ is countable.
\end{lemma}

\begin{proof}
    It is enough to show that there is $\alpha<\omega_1$ such that
    the only $\rho\in P(K)$ such that $\rho|C(K_\alpha)=\nu|C(K_\alpha)$
    is $\nu$. This is because, in this case, we can take a countable norm dense set $D$
    in $C(K_\alpha)$ ($C(K_\alpha)$ is separable as $K_\alpha$ is metrizable) and,
    if $\rho|D=\nu|D$, then $\rho|C(K_\alpha)=\nu|C(K_\alpha)$ and so $\rho=\nu$. It follows that 
    the weak$^*$ basic open sets of the form
    $$\left\{\theta\in P(K): |(\theta-\nu)(d_i)|<\frac{1}{n} \right\}$$
    for some $d_i\in D$ and $n\in \N$
    form a pseudobasis at $\nu$ in
    $B_{C(K)^*}$. By compactness of $B_{C(K)^*}$, finite intersections of elements of this pseudobasis
    form a basis at $\nu$ which is countable as required.

    So fix $\nu\in P(K)$ such that $\nu\perp\mu$. By Corollary \ref{corollary-singular-Baire} 
    there is a Baire set $E'\subseteq K$ such that
    $\mu(E')=0$ and $\nu(E')=1$. Since $E'$ is Baire, there is by Corollary \ref{Baire-set-depends} $\beta<\omega_1$
    such that $E'=\pi_{\beta, \omega_1}^{-1}[E]$ for some $E\subseteq K_\beta$ and 
    $\mu_\beta(E)=\mu(E')=0$ by (C3).
    By Haydon's condition there is $\beta\leq\alpha<\omega_1$ such that for every $\alpha\leq\gamma<\omega_1$
    we have $\pi_{\beta, \gamma}^{-1}[E]\cap F_\gamma=\emptyset.$ So, $\pi_{\alpha, \omega_1}|E'$ is
     injective by Lemma \ref{lemma-injective}. For $n\in \N$ let $E_n' \subseteq E'$ be compact sets, 
     such that $\nu(E_n') \geq 1 - \frac{1}{n}$ (which exist by regularity of $\nu$). 
     It follows that $(\pi_{\alpha,\omega_1}|E_n')^{-1}$ is well defined and continuous 
     for every $n \in \N$ as $\pi_{\alpha,\omega_1}|E_n'$ is an injective continuous 
     function on the compact space $E_n'$. We will now show that for all $n \in \N$
    \begin{align*}
        \{f|E_n': f\in C(K)\}=\{(g\circ\pi_{\alpha, \omega_1})|E_n': g\in C(K_\alpha)\}.
    \end{align*}
    Fix $n \in \N$ and let $f \in C(K)$, then $f \circ (\pi_{\alpha,\omega_1}|E_n')^{-1}$
     is a continuous function on compact $\pi_{\alpha,\omega_1}[E_n']$, and thus is 
     uniformly continuous and can be extended to $g \in C(K_\alpha)$. 
     Clearly, $(g\circ\pi_{\alpha, \omega_1})|E_n' = f|E_n'$. 
     The other inclusion is clear as $g\circ\pi_{\alpha, \omega_1} \in C(K)$ for every $g \in C(K_\alpha)$.

    Let $\rho\in P(K)$ be such that $\rho|C(K_\alpha)=\nu|C(K_\alpha)$. 
    We will now show that $\rho(E')=1$. Fix $\varepsilon > 0$ and 
    find $n \in \N$ such that $\nu(E_n') > 1-\varepsilon$. It follows 
    that $S = \pi_{\alpha,\omega_1}[E_n']$ is a compact subset of metrizable 
    $K_\alpha$, and thus is $G_\delta$. Hence, there are for $n \in \mathbb{N}$ 
    continuous functions $h_n \in C(K_\alpha)$ such that $h_n(K_\alpha) \subseteq [0,1]$, 
    $h_n|S = 1$ and such that $h_n \rightarrow \chi_S$ pointwise. By the dominated 
    convergence theorem and injectivity of $\pi_{\alpha,\omega_1}$ on $E_n'$ we get
    \begin{align*}
        \nu(h_n) = \int h_n \circ \pi_{\alpha,\omega_1} d \nu &\rightarrow \int \chi_S 
        \circ \pi_{\alpha,\omega_1} d \nu = \int \chi_{E_n'} d \nu = \nu(E_n') > 1-\varepsilon,\\
        \rho(h_n) = \int h_n \circ \pi_{\alpha,\omega_1} d \rho &\rightarrow \int \chi_S 
        \circ \pi_{\alpha,\omega_1} d \rho = \int \chi_{E_n'} d \rho = \rho(E_n').
    \end{align*}
    But $\nu|C(K_\alpha) = \rho|C(K_\alpha)$, and hence $\rho(E_n') > 1-\varepsilon$. 
    As $\varepsilon$ was arbitrary and $E_n' \subseteq E'$ for all $n \in \N$, $\rho(E')=1$.

    Let $f\in C(K)$ and $n \in \N$. Find $g_n \in C(K_\alpha)$ such that $f|E_n' = 
    (g_n \circ \pi_{\alpha,\omega_1})|E_n'$. Then
    \begin{align*}
        \int fd\rho&=\int_{K \setminus E_n'} fd\rho + \int_{E_n'} f d \rho=
        \int_{K \setminus E_n'} fd\rho + \int_{E_n'} (g_n \circ \pi_{\alpha,\omega_1}) d \rho, \\
        \int fd\nu&=\int_{K \setminus E_n'} fd\nu + \int_{E_n'} f d \nu=\int_{K \setminus E_n'} 
        fd\nu + \int_{E_n'} (g_n \circ \pi_{\alpha,\omega_1}) d \nu.
    \end{align*}
    Hence, again as $\nu|C(K_\alpha) = \rho|C(K_\alpha)$,
    \begin{align*}
        \left|\int fd\rho - \int fd\nu \right| = \left|\int_{K \setminus E_n'} fd\rho - \int_{K \setminus E_n'} fd\nu \right| \rightarrow 0
    \end{align*}
    as both $\rho(K \setminus E_n')$ and $\nu(K \setminus E_n')$ converge to zero. 
    This proves that $\rho=\nu$ as required.
\end{proof}
Recall  
that if every element of $P(K)$ has countable character, then $K$ does not carry a measure of uncountable type
(see footnote 2).

\section{Guessing sequences of measures with $\diamondsuit$}

Recall that a subset $A$ of $\omega_1$ is a 
club set (closed unbounded), if it is closed in the order topology (or equivalently under taking
suprema of subsets) and unbounded in $\omega_1$. A subset $B$ of $\omega_1$
 is stationary if it intersects every club set in $\omega_1$. Clearly, every club set is stationary and for a 
 club set (resp. stationary set) $A$, its subset of limit ordinals $\operatorname{Lim} \cap A$
  is a club set (resp. stationary set). For more details see e.g. \cite{kunen, jech}.

\begin{lemma} \label{club} Suppose that $K_\alpha\subseteq\{0,1\}^\alpha$ for $\alpha\leq\omega_1$
are compact such that $\pi_{\alpha, \beta}[K_\beta]=K_\alpha$ for every
$\alpha\leq\beta\leq\omega_1$ and $K_\lambda$ is the inverse limit of
$(K_\beta)_{\beta<\lambda}$ for all limit $\lambda \leq \omega_1$ as in Definition \ref{def-limits}. Put $K = K_{\omega_1}$.

   \noindent  Suppose that $\nu \in P(K)$ and $\nu_\alpha\in P(K)$ for every $\alpha<\omega_1$. Further suppose that
    $$\nu \in \overline{\{\nu_\alpha: \alpha<\omega_1\}}^{w^*},$$ where
    the closure is with respect to the weak$^*$ topology of $C(K)^*$. 
    Then there is a closed unbounded $C\subseteq\omega_1$ such that for every $\beta\in C$ we have
    $$\nu|C(K_\beta) \in \overline{\{\nu_\alpha|C(K_\beta): \alpha<\beta\}}^{w^*},$$
    where the closure is taken with respect to the weak$^*$ topology of $C(K_\beta)^*$.
\end{lemma}

\begin{proof}
    Let $C$ be the set of all such $\beta<\omega_1$.
    We will first prove that $C$ is closed. Let $(\beta_n)_{n\in \N}$ be an increasing sequence in $C$
    and $\beta = \sup_n \beta_n$. Fix $h_1, \dots h_k\in C(K_\beta)$ and $\varepsilon > 0$ defining a weak$^*$ 
    basic neighbourhood
    \begin{align*}
        \{\rho \in P(K_\beta): \; |\rho(h_i) - (\nu|C(K_\beta))(h_i)| <\varepsilon, \; i=1,\dots k\}
    \end{align*}
    of $\nu|C(K_\beta)$.
    By density of $\bigcup_{\gamma < \beta} C(K_\gamma)$ in $C(K_\beta)$ (Lemma \ref{limit}), 
    there are $g_1,\dots,g_k \in C(K_\gamma)$ for some $\gamma < \beta$ such 
    that $\|h_i-g_i \circ \pi_{\gamma,\beta}\|_\infty < \frac{\varepsilon}{4}$. 
    Fix $n \in \N$ such that $\gamma < \beta_n$. As $\beta_n \in C$, there is $\alpha < \beta_n$ 
    such that for all $i=1,\dots,k$
    \begin{align*}
        \left|(\nu|C(K_{\beta_n}))(g_i \circ \pi_{\gamma,\beta_n})-
        (\nu_\alpha|C(K_{\beta_n}))(g_i\circ \pi_{\gamma,\beta_n}) \right|
         < \frac{\varepsilon}{2}.
    \end{align*}
    Hence, for $i=1,\dots,k$
    \begin{align*}
        |(\nu_\alpha|C(K_\beta))(h_i) &-(\nu|C(K_\beta))(h_i)| \\
        &< |(\nu_\alpha|C(K_\beta))(g_i \circ \pi_{\gamma,\beta})-(\nu|C(K_\beta))(g_i\circ \pi_{\gamma,\beta})| +
         \frac{\varepsilon}{2} \\
        &= \left|(\nu_\alpha|C(K_{\beta_n}))(g_i \circ \pi_{\gamma,\beta_n})-
        (\nu|C(K_{\beta_n}))(g_i\circ \pi_{\gamma,\beta_n}) \right| 
        + \frac{\varepsilon}{2} \\
        &< \varepsilon,
    \end{align*}
    proving that $\beta \in C$.

    For unboundedness above an arbitrary $\beta_0<\omega_1$ we construct
    by induction an increasing sequence $(\beta_n)_{n\in \omega}$ in $\omega_1$
    such that for every $h_1, \dots h_k\in C(K_{\beta_n})$ and $\varepsilon>0$ there is
    $\alpha<\beta_{n+1}$ such that $|(\nu_\alpha|C(K_{\beta_n}))(h_i)-(\nu|C(K_{\beta_n}))(h_i)|<\varepsilon$ for $i = 1,\dots,k$.
    This can be done as it is sufficient to consider only $h_i$'s from countable norm dense subsets of $C(K_{\beta_n})$
     and rational $\varepsilon$'s as in the proof of closedness of $C$ above. Let $\beta=\sup_{n\in \omega}\beta_n$.
    Since $\bigcup_{n\in \omega}C(K_{\beta_n})$ is dense in $C(K_\beta)$, for every $h_1, \dots h_k\in C(K_{\beta})$
    and $\varepsilon>0$ we can find $g_1, \dots, g_k\in C(K_{\beta_n})$ for some $n\in \N$
    satisfying $\|g_i \circ \pi_{\beta_n,\beta}-h_i\|_\infty<\frac{\varepsilon}{4}$. By the construction there
     is $\alpha<\beta_{n+1}<\beta$
    such that $|(\nu_\alpha|C(K_{\beta_n}))(g_i)-(\nu|C(K_{\beta_n}))(g_i)|<\frac{\varepsilon}{2}$ for all $i = 1,\dots,k$.
    So for $i = 1,\dots,k$
    \begin{align*}
        |(\nu_\alpha|C(K_\beta))(h_i) &-(\nu|C(K_\beta))(h_i)| \\
        &< |(\nu_\alpha|C(K_\beta))(g_i \circ \pi_{\gamma,\beta})-(\nu|C(K_\beta))(g_i\circ \pi_{\gamma,\beta})| + 
        \frac{\varepsilon}{2} \\
        &= \left|(\nu_\alpha|C(K_{\beta_n}))(g_i \circ \pi_{\gamma,\beta_n})-(\nu|C(K_{\beta_n}))(g_i\circ \pi_{\gamma,\beta_n}) \right| + 
        \frac{\varepsilon}{2} \\
        &< \varepsilon,
    \end{align*}
    and $\beta \in C$, showing that $C$ is unbounded. 
\end{proof}

Recall Jensen's $\diamondsuit$ principle: there is a sequence $(S_\alpha)_{\alpha<\omega_1}$ such that
for every $X\subseteq\omega_1$ the set 
$$\{\alpha<\omega_1: X\cap\alpha=S_\alpha\}$$
is stationary in $\omega_1$. It is well known that $\diamondsuit$ is relatively consistent with {\sf ZFC} and implies
{\sf CH}. For more details see \cite{kunen} or \cite{jech}.

\begin{lemma} \label{diamond-measures}
    Assume $\diamondsuit$. There is a sequence $(M_\alpha)_{\alpha<\omega_1}$
    such that for every $\alpha<\omega_1$ the element $M_\alpha$ is a sequence $(M_{\alpha, n})_{n\in \omega}$
    of probability Radon measures on $\{0,1\}^\alpha$ and for every sequence
    $\rho = (\rho_\alpha)_{\alpha < \omega_1}$ of probability Radon measures on $\{0,1\}^{\omega_1}$ the set
    \begin{align*}
        S = \{\beta<\omega_1: & \{M_{\beta, n}: n\in \omega\}=\{\rho_\alpha|C(\{0,1\}^\beta): \alpha<\beta\} \\
        &\text{and } M_{\beta, 0}=\rho_0|C(\{0,1\}^\beta)\}
    \end{align*}
    is stationary in $\omega_1$.
\end{lemma}

\begin{proof}
    We use the following characterization of  $\diamondsuit$ (\cite[Theorem 2.7]{devlin}, cf. \cite[Lemma 6.4.]{dimension}):
    There exists a sequence $(f_\alpha)_{\alpha<\omega_1}$ of functions $f_\alpha:\alpha\rightarrow\alpha$
    such that for each $f:\omega_1\rightarrow\omega_1$
    the set $\{\alpha : f|\alpha=f_\alpha\}$ is stationary in $\omega_1$.

    First we establish some notation. For $\alpha < \omega_1$ let $w_\alpha \in C(\{0,1\}^{\omega_1})$, $w_\alpha(x) = 
    x(\alpha)$ be the coordinate function. For $F \in [\omega_1]^{<\omega}$ let 
    $w_F = \prod_{\alpha \in F} w_\alpha$ (we consider $w_\emptyset$ to be the constant $1$ function). 
    Then the linear span of $w_F$'s forms a subalgebra of $C(\{0,1\}^{\omega_1})$ 
    that contains constants and separates points of $\{0,1\}^{\omega_1}$, and thus is norm 
    dense in $C(\{0,1\}^{\omega_1})$ by the Stone-Weierstrass theorem. 
    Hence, every Radon measure $\nu \in M(\{0,1\}^{\omega_1})$ is determined by its values on 
    $w_F$'s for $F\in [\omega_1]^{<\omega}$. In the same way, for $\alpha < \omega_1$, the Radon 
    measure $\mu | C(\{0,1\}^\alpha)$ is determined by its values on $w_F$'s for $F\in [\alpha]^{<\omega}$.

    Hence, we can represent each probability Radon measure $\nu \in M(\{0,1\}^{\omega_1})$ by the function
    \begin{align*}
        \varphi_\nu : [\omega_1]^{<\omega} \rightarrow \mathbb{R}, \; \varphi_\nu (F) = \nu (w_F).
    \end{align*}
    Similarily, for $\beta < \omega_1$, probability Radon measures on $\{0,1\}^\beta$ can be represented as
     real valued functions on $[\beta]^{<\omega}$ and then $\nu | C(\{0,1\}^\beta)$ is represented by
    \begin{align*}
        \varphi_\nu |  [\beta]^{<\omega} : [\beta]^{<\omega} \rightarrow \mathbb{R}, \; \varphi_\nu (F) = \nu (w_F).
    \end{align*}
    If $\rho = (\rho_\alpha)_{\alpha < \omega_1}$ is a sequence whose elements are probability 
    Radon measures on $\{0,1\}^{\omega_1}$, we represent it by
    \begin{align*}
        \varphi_\rho: [\omega_1]^{<\omega} \times \omega_1 \rightarrow \mathbb{R}, \; \varphi_\rho (F,\alpha) = \rho_\alpha (w_F).
    \end{align*}
    Now, let $\Phi_1 : \omega_1 \rightarrow [\omega_1]^{<\omega} \times \omega_1$ be a bijection 
    such that for every limit $\alpha < \omega_1$ the restriction $\Phi_1 | \alpha$ is a bijection of $\alpha$ 
    and $[\alpha]^{<\omega} \times \alpha$\footnote{To see that this bijection exists note that for every 
    limit $\gamma < \omega_1$ there is a bijection 
    $\phi_\gamma: [\gamma,\gamma + \omega) \rightarrow 
    \left( [\gamma + \omega]^{< \omega} \times (\gamma + \omega) \setminus [\gamma]^{< \omega} \times \gamma \right)$. 
    We can take $\Phi_1 | [\gamma,\gamma + \omega) = \phi_\gamma$.}. 
    We also fix a bijection $\Phi_2 : \mathbb{R} \rightarrow \omega_1$
     (which exists as $\diamondsuit$ implies {\sf CH}). Finally, put
    \begin{align*}
        \psi_\rho = \Phi_2 \circ \varphi_\rho \circ \Phi_1 : \omega_1 \rightarrow \omega_1.
    \end{align*}
    Since for limit $\beta < \omega_1$ we know that $\Phi_1 | \beta$ is a bijection of $\beta$
     and $[\beta]^{<\omega} \times \beta$, $\psi_\rho | \beta$ is the representation of 
     $(\rho_\alpha|C(\{0,1\}^\beta))_{\alpha<\beta}$.

    Let $(f_\alpha)_{\alpha <\omega_1}$ be the diamond sequence mentioned in the beginning of the proof.
    Let $\alpha < \omega_1$ be a limit ordinal. If $f_\alpha = \psi_\rho | \alpha$ for some sequence 
    $\rho = (\rho_\alpha)_{\alpha < \omega_1}$ whose elements are Radon probability measures on $\{0,1\}^{\omega_1}$, 
    we set $M_\alpha = (M_{\alpha,n})_{n \in \omega}$ to be any enumeration of $\{\rho_\beta|C(\{0,1\}^\alpha)\}_{\beta <\alpha}$
     which satisfies $M_{\alpha,0} = \rho_0 | C(\{0,1\}^{\alpha})$.
    If $f_\alpha$ does not represent any such $\rho$, or if $\alpha < \omega_1$ is not a
     limit ordinal, we chose $M_\alpha$ to be any sequence of Radon probability measures on $\{0,1\}^\alpha$.

    It remains to show that the statement of the lemma is true for this choice of $(M_\alpha)_{\alpha < \omega_1}$. 
    Let $\rho = (\rho_\alpha)_{\alpha < \omega_1}$ be a sequence of Radon probability measures on $\{0,1\}^{\omega_1}$ 
    and let $S$ be as in the statement of the lemma. Set
    \begin{align*}
        T' = \{\alpha < \omega_1: \psi_\rho | \alpha = f_\alpha\}.
    \end{align*}
    Then $T'$ is stationary, and thus the subset $T = (T' \cap \operatorname{Lim})$ of its limit ordinals is also stationary. 
    But for $\alpha \in T$ we have that $(M_{\alpha,n})_{n \in \omega}$ is an enumeration 
    of $\{\rho_\beta|C(\{0,1\}^\alpha)\}_{\beta <\alpha}$ which satisfies 
    $M_{\alpha,0} = \rho_0 | C(\{0,1\}^{\alpha})$ and thus $\alpha \in S$. 
    Hence, $S$ is a superset of $T$ and is also stationary.
\end{proof}

\section{Single extension}

In this Section we show how to chose $\phi_\alpha$  to do the successor step from $K_\alpha$ to $K_{\alpha+1}$ for
$\alpha<\omega_1$ so that some topological conditions concerning closures are preserved (Lemma \ref{main-lemma}).

\begin{lemma}\label{half}
    Suppose that $\varepsilon > 0$, $n \in \N$ and $\{a_i: 1\leq i\leq n\}$ are non-negative real numbers smaller 
    than $\varepsilon$. Then there are disjoint sets $A, B$ such that $A\cup B=\{1, \dots, n\}$ 
    and
    $$\left|\sum_{i\in A}a_i-\frac{1}{2}\sum_{ i\leq n}a_i\right| = 
    \left|\sum_{i\in B}a_i-\frac{1}{2}\sum_{ i\leq n}a_i\right|< \frac{\varepsilon}{2}.$$
\end{lemma}

\begin{proof}
    First note that for any partition $A,B$ of $\{1, \dots, n\}$ we have
    $$\left|\sum_{i\in A}a_i-\frac{1}{2}\sum_{ i\leq n}a_i\right| = 
    \left|\sum_{ i\leq n}a_i - \sum_{i\in B}a_i-\frac{1}{2}\sum_{ i\leq n}a_i\right| = 
    \left|\sum_{i\in B}a_i-\frac{1}{2}\sum_{ i\leq n}a_i\right|.$$

    Let $k$ be the smallest element of $\{1,\dots,n\}$ such that $\sum_{i \leq k}
     a_i \geq \frac{1}{2} \sum_{i \leq n} a_i$. Then, as $0 \leq a_i < \varepsilon$ for all $1 \leq i \leq n$,
      either $$\left| \frac{1}{2} \sum_{i \leq n} a_i - \sum_{i \leq k} a_i \right| 
      < \frac{\varepsilon}{2} \;\; \text{ or } \;\; \left| \frac{1}{2} 
      \sum_{i \leq n} a_i - \sum_{i \leq k-1} a_i \right| < \frac{\varepsilon}{2}.$$
    If the first case holds set $A = \{1,\dots,k\}$, otherwise set $A = \{1,\dots,k-1\}$, 
    and finally let $B = \{1,\dots,n\} \setminus A$. Then clearly $A,B$ satisfy the conclusion of the lemma.
\end{proof}

\begin{lemma}\label{main-lemma} Suppose that  $\omega \leq \alpha<\omega_1$
and that a transfinite sequence $(K_\beta, \mu_\beta)_{\beta\leq \alpha}$ is
determined by a parameters $(F_\beta, \phi_\beta)_{\omega \leq \beta<\alpha}$. 
    Suppose that 
    \begin{enumerate}
    \item[(a)] $f_n^k, f^k$
    are Baire functions in $L_1(\mu_\alpha)$ for $n, k\in \N$.
    \item[(b)] $\nu^k, \nu_n^k$ are positive Radon measures on $K_\alpha$ such that $\nu^k(F_\alpha)=\nu_n^k(F_\alpha)=0$ 
    and $\nu^k, \nu_n^k\perp \mu_\alpha$ for every $n, k\in \N$.
     \item[(c)]  $f^k\mu_\alpha+\nu^k,
   f_n^k\mu_\alpha+\nu_n^k\in P(K_\alpha)$ for every $n, k\in\N$.
   \end{enumerate}
    Assume that for every $k\in\N$ we have
    $$f^k\mu_\alpha+\nu^k\in\overline{\{f_n^k\mu_\alpha+\nu_n^k: n\in \N\}}^{w^*},$$ 
    where the closure is taken with respect to the weak$^*$ topology in $C(K_\alpha)^*$.

    Then there is a continuous $\phi: (K_\alpha\setminus F_\alpha)\rightarrow \{0,1\}$ such that
    if $K_{\alpha+1}$ and $\mu_{\alpha+1}$ are extensions of $K_\alpha$ and $\mu_\alpha$ by $\phi$ respectively
    as defined in Definition \ref{def-ext}, 
    then for every $k\in \N$ we have
    $$(f^k\circ\pi_{\alpha, \alpha+1})\mu_{\alpha+1}+(\nu^k)^{\alpha+1}\in
    \overline{\{(f_n^k\circ\pi_{\alpha, \alpha+1})\mu_{\alpha+1}+(\nu_n^k)^{\alpha+1}: n\in \N\}}^{w^*},$$ 
    where the closure is taken with respect to the weak$^*$ topology in $C(K_{\alpha+1})^*$.
\end{lemma}

\begin{proof}
    For brevity we use the following notation for $k,n \in \N$:
    \begin{align*}
        \rho^k &= f^k\mu_\alpha+\nu^k, \\
        \rho^k_n &= f^k_n\mu_\alpha+\nu^k_n, \\
        \tilde{\rho}^k &= (f^k\circ\pi_{\alpha, \alpha+1})\mu_{\alpha+1}+(\nu^k)^{\alpha+1}, \\
        \tilde{\rho}^k_n &= (f^k_n\circ\pi_{\alpha, \alpha+1})\mu_{\alpha+1}+(\nu^k_n)^{\alpha+1}.
    \end{align*}
    Fix a decreasing clopen basis $(V_m)_{m\in \N}$  of $F_\alpha$ in $K_\alpha$ and let $V_0 = K_\alpha$.
    Note that if $L$ is compact totally disconnected, $\mathcal{F}$ is a family of clopen
     partitions of $L$ such that for every clopen partition of $L$ there is an element of $\mathcal{F}$
      which refines it, and if $\mu\in B_{C(L)^*}$,
    then sets of the form 
    $$\bigcap_{1\leq i\leq m} \left\{\nu \in B_{C(L)^*}: |(\nu-\mu)(W_i)| < \frac{1}{j} \right\}$$
    for $m, j\in \N$ and $(W_i)_{i \leq m} \in \mathcal{F}$ form a neighborhood basis at $\mu$ in $B_{C(L)^*}$.
     This is because if $\mu\in B_{C(L)^*}$ and a subbasic weak$^*$ neighborhood of $\mu$
     of the form $\{\nu \in B_{C(L)^*}: |(\nu-\mu)(f)| < \frac{1}{j}\}$ is given where $f\in C(K)$,
     we can find a simple continuous function $g$ with $\|g-f\|<\frac{1}{4j}$, $\|g\|\leq\|f\|$ assuming $n\in \N$
     values on clopen subsets $\{W_1, \dots W_k\}$ of $L$. Then $|(\nu-\mu)(W_i)| < \frac{1}{2\|f\|jk}$
      for all $1\leq i\leq k$ and  $\nu\in B_{C(L)^*}$
     implies $|(\nu-\mu)(f)|<\frac{1}{j}$ since
     $$|(\nu-\mu)(f)|\leq 2\|f-g\|+\sum_{1\leq i\leq k}|(\nu-\mu)(W_i)|\|g\|
     < \frac{1}{2j}+k\frac{1}{2\|f\|jk}\|g\|\leq \frac{1}{j}.$$

    We will consider, as our family $\mathcal{F}$, the family of clopen partitions of $L=K_{\alpha} \times \{0,1\}$ of the form 
    $$U_1\times\{0\}, \dots, U_m\times\{0\}, U_1\times\{1\}, \dots, U_m\times\{1\}$$ 
    for $m\in\N$,  where 
    $U_1, \dots, U_m$ is a clopen partition of $K_\alpha$. 
    Certainly such partitions refine all clopen partitions of $K_{\alpha}\times \{0,1\}$.

    We construct $\phi$ from pieces $\phi_s$ of clopen domains disjoint from $F_\alpha$ by induction on $s\in \N$ such that
    $$K_\alpha\setminus V_s\subseteq \operatorname{dom}(\phi_s)$$
    for all $s\in \N$.  Start with $\phi_0=\emptyset$ and fix a bijection
    $\Phi=(\Phi_1, \Phi_2, \Phi_3)$ with domain $\N$ onto the cartesian
     product $\mathcal{F} \times \N \times \N$ (here we use the metrizability of $K_{\alpha}\times \{0,1\}$).
    
    Given $s\in \N$  we will extend $\phi_s$ to $\phi_{s+1}$
    in such a way that there is $n\in\N$ such that no matter how we extend $\phi_{s+1}$ to $\phi$, 
    the induced extensions $K_{\alpha+1}$ and $\mu_{\alpha+1}$ of $K_\alpha$ and $\mu_\alpha$ by $\phi$ 
    as in Definition \ref{def-ext} will satisfy
    $$\left|\left(\tilde{\rho}^k-
    \tilde{\rho}^k_n\right)
    ((U_i\times \{l\}) \cap K_{\alpha+1})\right|<\frac{1}{j}\leqno (0)$$
    for all ${1\leq i\leq m}$ and all $l\in\{0,1\}$, 
    where $\Phi_1(s) = \{U_1 \times \{l\}, \dots, U_m \times \{l\}: \; l \in \{0,1\}\}$, $\Phi_2(s)=k$ and $\Phi_3(s)=j$.
    
    As a first preparatory step in defining $\phi_{s+1}$ note
    that we may assume that $K_\alpha\setminus V_{s+1}$ is already
    included in $\operatorname{dom}(\phi_s)$ (by extending $\phi_s$ continuously to some clopen
    set disjoint from $F_\alpha$ containing $K_\alpha\setminus V_{s+1}$).
    Similarly by the regularity of $\nu^k$ and the hypothesis (b) that $\nu^k(F_\alpha)=0$ we may assume that 
    $$\nu^k(K_\alpha\setminus \operatorname{dom}(\phi_s))<\frac{1}{48j} \leqno (1)$$
    and by the regularity of $f^k\mu_\alpha$ we may assume that 
    $$(f^k\mu_\alpha)(K_\alpha\setminus (\operatorname{dom}(\phi_s)\cup F_\alpha))<\frac{1}{48j}.\leqno (2)$$
    For $U_1,\dots,U_m$ as above define
    \begin{align*}
        W_i&=U_i\cap\phi_s^{-1}[\{0\}] &\text{ for } 1\leq i\leq m,& \\
        W_i&=U_{i-m}\cap\phi_s^{-1}[\{1\}] &\text{ for } m+1\leq i\leq 2m,& \\
        W_i&=U_{i-2m}\setminus \operatorname{dom}(\phi_s) &\text{ for } 2m+1\leq i\leq 3m.&
    \end{align*} 
    Note that if $\phi:K_\alpha\setminus F_\alpha\rightarrow\{0,1\}$ is any function
    which extends $\phi_s$ and $K_{\alpha+1}$, $\mu_{\alpha+1}$ are extensions of
     $K_\alpha$, $\mu_\alpha$ by $\phi$, we get by (C6) that for $1\leq i\leq 2m$, 
     $g\in L_1(\mu_\alpha)$, $\rho\in M(K_\alpha)$, $\rho(F_\alpha)=0$, and $l\in \{0,1\}$
    $$((g\circ\pi_{\alpha, \alpha+1})\mu_{\alpha+1}+\rho^{\alpha+1})((W_i\times\{l\}) \cap K_{\alpha+1})=
      \begin{cases}
        0 & \text{if $\phi_s|W_i\not=l$} \\
       (g\mu_{\alpha}+\rho)(W_i) & \text{if $\phi_s|W_i=l$}.
      \end{cases}
    $$
    It follows that for any $n\in \N$ that satisfies
    $$|(\rho^k-\rho^k_n)
    (W_i)|<\frac{1}{3j} \leqno (3)$$
    for $1\leq i\leq 2m$ we have 
    $$|(\tilde{\rho}^k - \tilde{\rho}^k_n)
    ((W_i\times \{l\}) \cap K_{\alpha+1})|<\frac{1}{3j}\leqno (4)$$
    for $l \in \{0,1\}$ as this value is either $0$ if $\phi_s|W_i\not=l$, or $|(\rho^k-\rho^k_n)
    (W_i)|$ if $\phi_s|W_i=l$.

    Aiming at proving (0) to take care of $W_i\times\{l\}$ for $2m+1\leq i\leq 3m$ and $l\in \{0,1\}$ 
    in a similiar manner to (4) we will use, among others,
    the fact  that $f^k\mu_\alpha$ is nonatomic (as it is absolutely continuous
     with respect to the measure $\mu_\alpha$ which is nonatomic by Lemma \ref{lemma-unc-type}).
     This nonatomicity will be used to find parts $W_i^l$ of $W_i$ for $l=0,1$ 
     which have $\rho_k^n$ measure equal to more or less half of the measure 
     of the corresponding $W_i$.  Then $\phi_{s+1}$ will be defined as having its value $l$ in
     most of $W_i^l$ to guarantee an estimate similar to (4). The motivation behind this
     is that since  the essential part $\mu_\alpha$ of $\rho^k$ is split into two equal parts on $F_\alpha$ when we pass
     from $K_\alpha$ to $K_{\alpha+1}$ we try to split the measures $\rho^n_k$ the same way
     on neighbourhoods of $F_\alpha$.

    For each $1\leq i\leq m$ 
    we find pairwise disjoint clopen sets $X_i^r$ for $1\leq r\leq t$ for some $t\in \N$ 
    such that $\bigcup_{1\leq r\leq t}X_i^r=W_{2m+i}$ and 
    $$(f^k\mu_\alpha)(X_i^r)<\frac{1}{48j}$$
    for every $1\leq r\leq t$. By (1) this gives
    $$\rho^k (X_i^r)<\frac{1}{24j}\leqno(5)$$
    for every $1\leq r\leq t$.

    Use the hypothesis of the lemma to find $n\in \N$ such that
    $$|(\rho^k - \rho^k_n)
    (W_i)|<\frac{1}{3j}\leqno(6)$$
    for all ${1\leq i\leq 2m}$ and
    $$|(\rho^k - \rho^k_n)
    (X_i^r)|<\frac{1}{24jt}\leqno(7)$$
    for all ${1\leq i\leq m}$ and $1\leq r\leq t$. 
    By (5) and (7) it follows that for these $i$ and $r$ we have
    $$\rho^k_n(X_i^r \setminus F_\alpha) \leq \rho^k_n(X_i^r) < \frac{1}{12j}.\leqno(8)$$
    By Lemma \ref{half} this allows us to select for $i \leq m$ a partition of $\{1, \dots, t\}$ into sets $A_i$ and $B_i$
    such that for $W_i^0=\bigcup_{r\in A_i}X_i^r$, $W_i^1=\bigcup_{r\in B_i}X_i^r$ and $l \in \{0,1\}$ we have
    $$\left|\rho^k_n(W_i^l \setminus F_\alpha)-
    \frac{1}{2}\rho^k_n (W_i \setminus F_\alpha)\right|<\frac{1}{24j}.\leqno (9)$$
     
    Find, by regularity, a clopen $Z\subseteq K_\alpha$
    such that $F_\alpha\subseteq Z$, $Z\cap \operatorname{dom}(\phi_s)=\emptyset$ and 
    $$\rho^k_n(Z \setminus F_\alpha)<\frac{1}{48j}.\leqno(10)$$ 
    Finally define $\phi_{s+1}$
    by extending $\phi_s$ to 
    $$\operatorname{dom}(\phi_s)\cup\bigcup_{2m+1\leq i\leq 3m}(W_i\setminus Z)$$
    by putting $\phi_{s+1}|(W_i^l\setminus Z)=l$ for $2m+1 \leq i \leq 3m$ and $l\in \{0,1\}$.
    
    Now we will prove that $\phi_{s+1}$ works. The function $\phi_{s+1}$ is continuous
    (the sets $W_i \setminus Z$, $2m+1 \leq i \leq 3m$, 
    and $\operatorname{dom}(\phi_s)$ form a clopen
    partition of $\operatorname{dom}(\phi_{s+1})$ and $\phi_{s+1}$ 
    is continuous on each of these sets). Let $\phi$ be any extension of 
    $\phi_{s+1}$ and suppose that $K_{\alpha+1}$ and $\mu_{\alpha+1}$ 
    are extensions of $K_\alpha$ and $\mu_\alpha$ by $\phi$ as in Definition \ref{def-ext}.
    Fix ${2m+1\leq i\leq 3m}$ and $l\in \{0,1\}$.
    
    To obtain estimate as in (4) for $2m< i\leq 3m$ we will separately approximate
    the measures $\tilde\rho^k$ and $\tilde\rho^k_n$ on the sets $(W_i\times \{l\})\cap K_\alpha+1$
    by  considering several of their subsets on which we already know approximately
    the values of the measures either by reducing them to some measures 
    on $K_\alpha$ or exploiting the facts that their values are small. First we approximate the measures $\tilde\rho^k$:
    we have
    \begin{align*} \tag{11}
        \tilde{\rho}^k((W_i\times\{l\})\cap K_{\alpha+1})
        =&  \tilde{\rho}^k((W_i \cap F_\alpha)\times \{l\})   \\
        &+\tilde{\rho}^k(((W_i \setminus Z)\times \{l\}) \cap K_{\alpha+1})\\
        &+\tilde{\rho}^k(((Z\cap W_i \setminus F_\alpha)\times \{l\})\cap K_{\alpha+1}).
    \end{align*}
    Let us deal with the three terms of the right hand side of (11) term by term. For the first term by (C7) 
    of Definition \ref{def-general} and the hypothesis (b) that $\nu^k(F_\alpha) = 0$ we have
    \begin{align*}
        \tilde{\rho}^k((W_i \cap F_\alpha)\times \{l\}) &= ((f^k\circ\pi_{\alpha, \alpha+1})
        \mu_{\alpha+1}) ((W_i \cap F_\alpha)\times \{l\}) \\
        &= \frac{1}{2} (f^k \mu_\alpha)(F_\alpha\cap W_i) = \frac{1}{2} \rho^k(F_\alpha\cap W_i).
    \end{align*}
    For the second term by the choice of $\phi_{s+1}$ and (C6) of Definition \ref{def-general} we have
    \begin{align*}
        \tilde{\rho}^k(((W_i \setminus Z)\times \{l\}) \cap K_{\alpha+1})
        = \rho^k (W_i^l \setminus Z)
        = \frac{1}{2}\rho^k(W_i \setminus F_\alpha) + c_1,
    \end{align*}
    where
    \begin{align*}
        |c_1| = \left|\rho^k (W_i^l \setminus Z) - \frac{1}{2}\rho^k(W_i \setminus F_\alpha) \right|
         \overset{(1,2)}{\leq} \frac{1}{24j}
    \end{align*} 
    as both $W_i \setminus F_\alpha$ and $W_i^l \setminus Z$ do not intersect $\operatorname{dom} (\phi_s)$
    and both of the terms inside the absolute value are non-negative.
    Now let us move to the third term of the right hand side of (11)
    We set
    \begin{align*}
       c_2= \tilde{\rho}^k(((Z \setminus F_\alpha)\times \{l\})\cap K_{\alpha+1}) \geq 
       \tilde{\rho}^k(((Z\cap W_i \setminus F_\alpha)\times \{l\})\cap K_{\alpha+1}),
    \end{align*}
    then
    \begin{align*}
        |c_2| &= \left| ((f^k\circ\pi_{\alpha, \alpha+1})\mu_{\alpha+1}+
        (\nu^k)^{\alpha+1})(((Z \setminus F_\alpha)\times \{l\})\cap K_{\alpha+1}) \right| \\
        & \leq \left| (f^k \mu_\alpha) (Z \setminus F_\alpha) \right| + \left| \nu^k 
        (Z \setminus F_\alpha) \right| \overset{(1,2)}{\leq} \frac{1}{24j}
    \end{align*}
    as $Z \setminus F_\alpha$ does not intersect $\operatorname{dom}(\phi_{s})$. 
    Altogether, we obtain the following approximation of the value in (11)
    \begin{align*}\tag{12}
        | \tilde{\rho}^k((W_i\times\{l\}) & \cap K_{\alpha+1})-\frac{1}{2} \rho^k (W_i) |\\
        &=\left|\tilde{\rho}^k((W_i\times\{l\})\cap K_{\alpha+1})- \frac{1}{2} \rho^k (F_\alpha\cap W_i) - \frac{1}{2}\rho^k(W_i \setminus F_\alpha) \right|\\
        &\leq |c_1|+|c_2|\leq \frac{1}{12j}.
    \end{align*}
    Now we approximate
    the measures $\tilde\rho^k_n$  on the sets $(W_i\times \{l\})\cap K_{\alpha+1}$
    by considering several of their subsets on which we already know approximately
    the values of the measures:
    \begin{align*} \tag{13}
        \tilde{\rho}^k_n ((W_i\times\{l\})\cap K_{\alpha+1}) =&  \tilde{\rho}^k_n((W_i \cap F_\alpha)\times \{l\})   \\
        &+\tilde{\rho}^k_n(((W_i \setminus Z)\times \{l\}) \cap K_{\alpha+1})\\
        &+\tilde{\rho}^k_n(((Z\cap W_i \setminus F_\alpha)\times \{l\})\cap K_{\alpha+1}).
    \end{align*}
    
    Let us deal with the three terms
    of the right hand side of (13) again term by term. For the first term again, by (C7) of Definition \ref{def-general}
    and the hypothesis (b)  that $\nu^k_n(F_\alpha) = 0$ we have
    \begin{align*}
        \tilde{\rho}^k_n((W_i \cap F_\alpha)\times \{l\}) = \frac{1}{2} \rho^k_n(F_\alpha\cap W_i).
    \end{align*}
    For the second term by the choice of $\phi_{s+1}$ and (C6) of Definition \ref{def-general} we have
    \begin{align*}
        \tilde{\rho}^k_n(((W_i \setminus Z)\times \{l\}) \cap K_{\alpha+1}) &= \rho^k_n(W_i^l \setminus Z) \\
        &= \rho^k_n(W_i^l \setminus F_\alpha) - \rho^k_n(Z \setminus F_\alpha)\\
        &= \frac{1}{2}\rho^k_n(W_i \setminus F_\alpha) + c_3
    \end{align*}
    where
    \begin{align*}
        |c_3| &= \left| \rho^k_n(W_i^l \setminus F_\alpha) - \rho^k_n(Z \setminus F_\alpha) - 
        \frac{1}{2}\rho^k_n(W_i \setminus F_\alpha) \right| \\
        &\leq  \left| \rho^k_n(W_i^l \setminus F_\alpha) - \frac{1}{2}\rho^k_n(W_i \setminus F_\alpha) \right|
         + \left| \rho^k_n(Z \setminus F_\alpha) \right|   \overset{(9,10)}{<} \frac{1}{16j}.
    \end{align*}    
   Considering the third term we will write
    \begin{align*}
        c_4=
        \tilde{\rho}^k_n(((Z \cap W_i\setminus F_\alpha)\times \{l\})\cap K_{\alpha+1}).
    \end{align*}
    Then
    \begin{align*}
        |c_4| \leq \left| \rho^k_n (Z \setminus F_\alpha) \right| \overset{(10)}{<} \frac{1}{48j}.
    \end{align*}
    
    Altogether, (13) gives
    \begin{align*}\tag{14}
        |\tilde{\rho}^k_n((W_i&\times\{l\})\cap K_{\alpha+1})-\frac{1}{2} \rho^k_n(W_i)| \\
        &= \left|\tilde{\rho}^k_n((W_i\times\{l\})\cap K_{\alpha+1}) - \frac{1}{2} \rho^k_n(F_\alpha\cap W_i) -
        \frac{1}{2}\rho^k_n(W_i \setminus F_\alpha) \right| \\
        &\leq |c_3|+|c_4|\leq \frac{1}{12j}.
    \end{align*}
    
    Finally using the choice of $\phi$ (as any  extension of $\phi_{s+1}$) and the choice of $n$ in (6)
    we have  by (4)
    $$|(\tilde{\rho}^k - \tilde{\rho}^k_n)
    ((W_i\times \{l\}) \cap K_{\alpha+1})|<\frac{1}{3j}$$
    for $1\leq i\leq 2m$ and $l \in \{0,1\}$ and 
    by (12) and (14),
    \begin{align*}
        |(\tilde{\rho}^k - \tilde{\rho}^k_n)
        ((W_i\times \{l\}) \cap K_{\alpha+1})| < \frac{1}{6j}+\frac{1}{12j}+\frac{1}{12j} \leq \frac{1}{3j}
    \end{align*}
    for $2m+1\leq i\leq 3m$ and $l \in \{0,1\}$.
    
    So, as $U_i = W_i \cup W_{m+i} \cup W_{2m+i}$, we get
    $$|(\tilde{\rho}^k - \tilde{\rho}^k_n)
    ((U_i\times \{l\}) \cap K_{\alpha+1})|< \frac{1}{j} \leqno{(15)}$$
    for $1\leq i \leq m$ and $l \in \{0,1\}$ as required.
    
    We define $\phi = \bigcup_{s \in \N} \phi_s$. Then $\operatorname{dom}\phi = 
    K_\alpha \setminus F_\alpha$ as $\operatorname{dom}\phi_s \supseteq K_\alpha \setminus V_s$ for all $s \in \N$.
    The function $\phi$ is continuous as for $l \in \{0,1\}$ we have $\phi^{-1}[\{l\}] = \bigcup_{s \in \N} \phi_s^{-1}[\{l\}]$
     and $\phi_s$ were continuous and had clopen domains.
    Further, for any clopen partition $\{U_1, \dots, U_m\}$ of $K_\alpha$, and for all $k,j \in \N$, there is $s \in \N$ 
     such that $\Phi_1(s) = \{U_1 \times \{l\}, \dots, U_m \times \{l\}: \; l \in \{0,1\}\}$, $\Phi_2(s)=k$ and $\Phi_3(s)=j$, and thus by the
      above there is $n \in \N$ such that (15) is satisfied
    for $1\leq i \leq m$ and $l \in \{0,1\}$. As $\{U_1, \dots, U_m\}$ and $j \in \N$ were arbitrary, 
    we have that $$\tilde{\rho}^k \in \overline{\{\tilde{\rho}^k_n : \; n \in \N\}}^{w^*}.$$
    As $k \in \N$ was arbitrary, the lemma follows.
\end{proof}

\section{Constructing the space}

Before we proceed with the construction of our space $K$ let us recall that if $\rho \in P(K)$ is decomposed into $\rho = f \mu + \nu$ for $\nu \perp \mu$, then both $f \mu$ and $\nu$ are positive.

\noindent{\bf Proof  of Theorem \ref{main}:}
\begin{proof}
As we assume $\diamondsuit$ we have the continuum hypothesis {\sf CH} as well.
Fix 
$$\{(M_{\alpha, n})_{n\in \omega}: \alpha<\omega_1\}$$ as in Lemma \ref{diamond-measures}.

We construct $K\subseteq\{0,1\}^{\omega_1}$ and a Radon 
probability measure $\mu$ on $K$ as $K_{\omega_1}$ and $\mu_{\omega_1}$
for  a transfinite sequence $(K_\alpha, \mu_\alpha)_{\alpha\leq \omega_1}$ 
determined by parameters $(F_\alpha, \phi_\alpha)_{\omega \leq \alpha<\omega_1}$ as in Definition \ref{def-general}.
Conditions (C1)-(C3) give us the construction of $K_\alpha$ and $\mu_\alpha$ for finite and limit ordinals
 $\alpha \leq \omega_1$.
Conditions (C4)-(C7) applied to $\phi_\alpha$ obtained from Lemma 
\ref{main-lemma} with the sequences of measures given by the $\diamondsuit$-sequence and $F_\alpha$
chosen to take care of Haydon's condition
give the construction of $K_{\alpha+1}$ and $\mu_{\alpha+1}$ from 
$K_\alpha$ and $\mu_\alpha$ for $\omega \leq \alpha < \omega_1$. 

More precisely, we construct
by transfinite recursion on $\omega \leq \alpha < \omega_1$ 
\begin{enumerate}
\item $K_\alpha$ and $\mu_\alpha$ as in Definition \ref{def-general}  determined by parameters
$(F_\beta, \phi_\beta)_{\omega \leq \beta<\alpha}$.
    \item An enumeration $(N_{\alpha, \beta})_{\beta<\omega_1}$ of all $\mu_\alpha$-null
     Borel subsets of $K_\alpha$.
    \item Baire functions  $f_{\alpha}, f_{\alpha, n}$  in $L_1(\mu_\alpha)$ for $n\in \N$
    and Radon measures $\nu_\alpha$, $\nu_{\alpha, n}$ for $n\in \N$ such that
    $f_{\alpha}\mu_\alpha+\nu_{ \alpha}, f_{\alpha, n}\mu_\alpha+\nu_{ \alpha, n}$
    are all Radon probability measures on $K_\alpha$ and  $\nu_{\alpha}, \nu_{\alpha, n}\perp\mu_\alpha$ 
    are such that $$f_{\alpha}\mu_\alpha+\nu_{\alpha} \in \overline{\{f_{\alpha, n}\mu_\alpha+\nu_{\alpha, n}: 
    {n\in \N}\}}^{w^*}$$ where the closure is taken with respect to the weak$^*$ topology on $M(K_\alpha)$.
    
\end{enumerate}
Additional properties which we prove inductively on $\alpha<\omega_1$ are (C1)-(C7) and:
\begin{enumerate}
    \item[(4)]  Whenever $(M_{\alpha, n})_{n\in \N}$ consists of Radon probability measures on $\{0,1\}^\alpha$ which are
    all concentrated on $K_\alpha$ and   $M_{\alpha, 0} \in \overline{(M_{\alpha, n})_{n>0}}^{w^*}$, 
    where the closure is taken with respect to the weak$^*$ topology on $M(K_\alpha)$, 
    then $$f_{\alpha}\mu_\alpha+\nu_\alpha=M_{\alpha, 0}\ 
    \hbox{ and }\  f_{\alpha, n}\mu_\alpha+\nu_{\alpha, n}=M_{\alpha, n}$$
    for all $n\in \N$. 
    \item[(5)] $F_{\alpha'}$ is disjoint from 
    $\pi_{\beta, {\alpha'}}^{-1}[N_{\beta, \gamma}]$ for every $\beta, \gamma\leq\alpha'<\alpha$.
    \item[(6)]     $(\nu_{\beta})^{\alpha'}(F_{\alpha'})=(\nu_{\beta, n})^{\alpha'}(F_{\alpha'})=0$
    for all $n\in \N$ and $\beta\leq\alpha'<\alpha$.
    \item [(7)] For every $\beta<\alpha \leq \omega_1$
    $$(f_{\beta}\circ \pi_{\beta, \alpha})\mu_\alpha+(\nu_{\beta})^\alpha \in
     \overline{\{(f_{\beta, n}\circ \pi_{\beta, \alpha})\mu_\alpha+(\nu_{\beta, n})^\alpha :\; n \in \N\}}^{w^*}$$
    where the closure is taken with respect to the weak$^*$ topology on $M(K_\alpha)$.
    
\end{enumerate}

Let us see that such recursive constructions can be carried
out. Suppose that we have constructed all the objects as in (1)-(3) that satisfy (1)-(7) for all $\alpha'<\alpha$ and (C1)-(C7)
of Definition \ref{def-general} for $\delta=\alpha$.

First assume that
$\alpha=\alpha'+1$ for some $\alpha'<\omega_1$.  To construct $K_\alpha=K_{\alpha'+1}$
as in (C6) and $\mu_\alpha=\mu_{\alpha'+1}$ as in (C7)
we need to specify the new parameters $F_{\alpha'}$ and $\phi_{\alpha'}$ as in (C4) and (C5) (as in
Definition \ref{def-general} in the construction of $K_\delta$ we specify the values of the parameters only for 
ordinals less than $\delta$).

First we show how to choose $F_{\alpha'}$ such that (C4),  (5) and (6) are satisfied. 
Since by the inductive hypothesis (3) the measures $\nu_\beta, \nu_{\beta, n}$, $n \in \N$, are 
singular with respect to $\mu_\beta$ respectively for
every $\beta\leq\alpha'$, there are Borel $G_\beta, G_{\beta,n}\subseteq K_{\beta}$
for $\beta\leq\alpha'$ and $n \in \N$ such that 
$$\mu_\beta(G_\beta)=\mu_\beta(G_{\beta, n})=\nu_{\beta, n}(K_\beta\setminus G_{\beta, n})=
\nu_{\beta}(K_\beta\setminus G_{\beta})=0.$$
It follows from (C3) of Definition \ref{def-general} that the preimages
of the  sets $G_\beta$, $G_{\beta, n}$, $N_{\beta, \gamma}$ for
$\beta, \gamma\leq \alpha'$ and $n \in \N$ under $\pi_{\beta, \alpha'}$ are $\mu_{\alpha'}$-null.
So by taking a compact subset $F_{\alpha'}'$ which is disjoint with
all these (countably many) sets and such that $\mu_{\alpha'}(F_\alpha') > \frac{3}{4}$ 
(by regularity of $\mu_{\alpha'}$) we can ensure that (C4), (5) and (6) are satisfied for $F_{\alpha'}'$ instead of $F_{\alpha'}$
with the exception of $F_{\alpha'}'$ not being nowhere dense as required in (C4).
What remains is to take a closed nowhere dense subset $F_{\alpha'}$ of $F_{\alpha'}'$ such 
that $\mu_{\alpha'}(F_{\alpha'}) \geq \frac{3}{4}$.
To do this, let $(U_n)_{n \in \N}$ be an open basis of the compact metrizable space $K_{\alpha'}$. 
As $\mu_{\alpha'}$ is atomless by Lemma \ref{lemma-unc-type}, there is for each $n \in \N$ an 
open subset $U_n'$ of $U_n$ such that $\mu_{\alpha'}(U_n') < 2^{-n} (\mu_{\alpha'}(F_{\alpha'}') - \frac{3}{4})$.
Let $F_{\alpha'} = F_{\alpha'}' \setminus \bigcup_{n \in \N} U_n'$. Then $F_{\alpha'}$ is clearly closed, 
is nowhere dense as it contains no $U_n$ as a subset and $\mu_{\alpha'}(F_{\alpha'}) \geq \frac{3}{4}$. 
Hence, $F_{\alpha'}$ satisfies (C4), (3), (5) and (6) as required.

Note that the hypothesis of Lemma \ref{main-lemma} for the measures
$f_{\beta}\circ\pi_{\beta, \alpha'}\mu_{\alpha'}+(\nu_\beta)^{\alpha'}$,
$f_{\beta, n}\circ\pi_{\beta, \alpha'}\mu_{\alpha'}+(\nu_{\beta, n})^{\alpha'}$ for $\beta\leq\alpha'$ and $n \in \N$,
and $F_{\alpha'}$
 is satisfied by conditions (3) and (6)
of the inductive hypothesis   at $\alpha'$ and the construction of $F_{\alpha'}$.
So we use Lemma \ref{main-lemma}
to  construct $\phi_{\alpha'}$ and following (C4)-(C7) the extensions
$K_\alpha=K_{\alpha'+1}$  and $\mu_{\alpha}=\mu_{\alpha'+1}$
by $\phi_{\alpha'}$ which satisfy this Lemma.
By Lemma \ref{main-lemma}
 we immediately obtain  that (7) is satisfied at $\alpha$.

Now to  make sure that (4) is satisfied we check  if 
its hypothesis  is satisfied, that is if $(M_{\alpha, n})_{n\in \N}$ consists of probability 
measures on $\{0,1\}^\alpha$ which are
all concentrated on $K_\alpha$ and the  weak$^*$ closure of  $(M_{\alpha, n})_{n>0}$ 
contains the probability measure $M_{\alpha, 0}$.
If so, we declare $f_{\alpha}\mu_\alpha+\nu_\alpha=M_{\alpha, 0}$ 
and $f_{\alpha, n}\mu_\alpha+\nu_{\alpha, n}=M_{\alpha, n}$ for $n \in \N$.
Otherwise we declare $f_{\alpha}\mu_\alpha+\nu_\alpha$
and $f_{\alpha, n}\mu_\alpha+\nu_{\alpha, n}$, $n \in \N$, to be  all the same probability measure
on $K_\alpha$. This way we have (3) and (4) at $\alpha$.
Condition (2) can be easily obtained since we have {\sf CH}.

This completes checking that
making the successor step is possible preserving all conditions (1)-(7) and (C1)-(C7).

Now suppose that $\alpha$ is a limit ordinal.  We construct $K_\alpha$, $\mu_\alpha$
following (C1) of Definition \ref{def-general} obtaining (1).  Condition (2)
is taken care by {\sf CH}.  Construction of objects mentioned in (3) taking
care of (4) is done the same way as in the successor case given $K_\alpha$. For limit $\alpha$
conditions (5) and (6) follow
from the inductive hypothesis (5) and (6).
Condition (7) immediately follows
from (7) for $\alpha'<\alpha$ because $\bigcup_{\alpha'<\alpha}C(K_{\alpha'})$ is dense in
$C(K_\alpha)$ by  Lemma \ref{limit}, so if a bounded net of measures converges
on each element of this set, it converges on every element of $C(K_\alpha)$ \footnote{Indeed, let $(\rho_\lambda)_\lambda$, $\rho$ be measures bounded by constant $K > 0$ such that for each $f \in \bigcup_{\alpha'<\alpha}C(K_{\alpha'})$ we have $\rho_\lambda(f) \rightarrow \rho(f)$. For $\epsilon > 0$ and $g \in C(K_\alpha)$ we find $f \in \bigcup_{\alpha'<\alpha}C(K_{\alpha'})$ with $\|f-g\| < \frac{\epsilon}{4K}$ and $\lambda_0$ such that for all $\lambda > \lambda_0$ we have $|\rho_\lambda(f) - \rho(f)| < \frac{\epsilon}{2}$. Then for $\lambda > \lambda_0$ $|\rho_\lambda(g) - \rho(g)| \leq |\rho_\lambda(f) - \rho(f)| + |\rho_\lambda(g-f) - \rho(g-f)| < \epsilon + 2K \|g-f\| < \epsilon$.}.

This completes the construction of $(K_\alpha)_{\alpha \leq \omega_1}$ and 
$(\mu_\alpha)_{\alpha\leq \omega_1}$ which satisfies (1)-(7) and (C1)-(C7).  Note
that Haydon's condition is satisfied by (5). 
In particular,  by Lemma \ref{lemma-unc-type} the measure  $\mu$ is of uncountable type.

Now we prove that $P(K)$ has  countable tightness. First we note that
the weight and so the character of  $P(K)$ in the weak$^*$ topology is not bigger than $\omega_1$.
Indeed, as $K\subseteq \{0,1\}^{\omega_1}$, it has weight not bigger than $\omega_1$
and so $C(K)$ has density not bigger than $\omega_1$ and hence the dual ball
of $C(K)$ and its subspace $P(K)$ has weight not bigger than $\omega_1$.

Hence, we need to show that whenever
a point $f_0\mu+\nu_0$ of $P(K)$  is in the closure of a set 
$\{f_\xi\mu+\nu_\xi: 0<\xi<\omega_1\}\subseteq P(K)$, then
there is $\alpha<\omega_1$ such that
$f_0\mu+\nu_0$ is in the closure of 
$\{f_\xi\mu+\nu_\xi: 0<\xi<\alpha\}$. Here $f_\xi\in L_1(\mu)$ 
and $\nu_\xi$'s are Radon measures on $K$ which are singular
with respect to $\mu$. By Proposition \ref{baire-L1} we may assume that each
$f_\xi$, $\xi <\omega_1$, is a Baire function and so depends (in the sense of
Definition \ref{def-depends}) on some $\gamma_\xi<\omega_1$ by Proposition \ref{baire-depends}.
Let $\beta_\xi<\omega_1$ be such that 
$\nu_\xi|C(K_{\beta'})=(\nu_\xi|C(K_{\beta_\xi}))^{\beta'}$ for
every $\beta_\xi\leq\beta'\leq\omega_1$ which exists by Proposition \ref{singular-collapses}.
Let $\alpha_\xi=\max\{\gamma_\xi,\beta_\xi\}$ for every $\xi<\omega_1$.

Let $D\subseteq\omega_1$ be a closed and unbounded set
in $\omega_1$ such that for every $\alpha\in D$ for every $\beta<\alpha$ we have $\alpha_\beta<\alpha$\footnote{This
is a standard closure argument: to get such a $D$ first note that it is clearly closed; to prove that it is unbounded,
above an arbitrary $\gamma=\gamma_0<\omega_1$ construct $\gamma_1<\dots \gamma_n<\omega_1$ for
$n\in \N$ such that $\alpha_\beta<\gamma_{n+1}$ for every $\beta<\gamma_n$. This can
be achieved since we have only countably many $\beta<\gamma_n$. Now $\sup_{n\in \N}\gamma_n$ belongs
to $D$.}.
Let $C\subseteq\omega_1$ be the closed and unbounded subset of $\omega_1$ as in
Lemma \ref{club} taken for $f_0 \mu + \nu_0$ and $\{f_\xi\mu+\nu_\xi: 0<\xi<\omega_1\}$. Let $S\subseteq\omega_1$ be a stationary set
as in Lemma \ref{diamond-measures} for $\rho_\xi=f_\xi\mu+\nu_\xi$, $\xi < \omega_1$. 
Take $\alpha\in C\cap D\cap S$ ($C\cap D$ is a club as $C, D$ are clubs, and $S$ intersects every club). We have
\begin{enumerate}
\item[(a)]  $(f_0\mu+\nu_0)|C(K_\alpha)\in \overline{\{(f_\xi\mu+\nu_\xi)|C(K_\alpha): 0<\xi<\alpha\}}^{w^*}$ (as $\alpha\in C$),
\item[(b)] $f_\xi$ depends on $\alpha$ and
$\nu_\xi|C(K_{\beta'})=(\nu_\xi|C(K_{\alpha}))^{\beta'}$
 for all $\xi<\alpha$  and $\alpha\leq\beta'\leq\omega_1$ (as $\alpha\in D$),
\item[(c)]  $(M_{\alpha, n})_{n\in\omega}$ is $\{(f_\xi\mu+\nu_\xi)|C(K_\alpha):\xi<\alpha\}$ and 
$M_{\alpha,0} = (f_0\mu+\nu_0)|C(K_\alpha)$ (as $\alpha\in S$).
\end{enumerate}
So by (a) and (c), the condition of item (4) of the construction at stage $\alpha$ is satisfied. So
by (7)  and (b) we obtain that $f_0\mu+\nu_0 \in \overline{\{f_\xi\mu+\nu_\xi: 0<\xi<\alpha\}}^{w^*}$ as required.
\end{proof}

\bibliographystyle{amsplain}

\end{document}